\documentclass[a4paper, 10pt, notitlepage]{article}

\usepackage{amsthm, amsmath, amssymb, latexsym}
\usepackage{mathrsfs,cite}
\usepackage[multiple]{footmisc}
\usepackage{graphicx}

\theoremstyle{plain}
\newtheorem{theorem}{Theorem}
\newtheorem{lemma}{Lemma}

\newtheorem{proposition}{Proposition}

\theoremstyle{definition}
\newtheorem{definition}{Definition}

\theoremstyle{remark}
\newtheorem{remark}{Remark}

\DeclareMathOperator{\dom}{dom}
\DeclareMathOperator{\diag}{diag}
\DeclareMathOperator{\esssup}{esssup}
\DeclareMathOperator{\trace}{Tr}
\DeclareMathOperator{\proj}{Pr}

\author{M.V. Dolgopolik\footnote{Institute for Problems in Mechanical Engineering of the Russian Academy of Sciences,
Saint Petersburg, Russia}}
\title{Exact augmented Lagrangians for constrained optimization problems in Hilbert spaces II: Applications}

\begin{document}

\maketitle

\begin{abstract}
This two-part study is devoted to the analysis of the so-called exact augmented Lagrangians, introduced by Di Pillo and
Grippo for finite dimensional optimization problems, in the case of optimization problems in Hilbert spaces. In the
second part of our study we present applications of the general theory of exact augmented Lagrangians to several
constrained variational problems and optimal control problems, including variational problems with additional
constraints at the boundary, isoperimetric problems, problems with nonholonomic equality constraints (PDE constraints),
and optimal control problems for linear evolution equations. We provide sufficient conditions for augmented Lagrangians
for these problems to be globally/completely exact, that is, conditions under which a constrained variational
problem/optimal control problem becomes equivalent to the problem of unconstrained minimization of the corresponding
exact augmented Lagrangian in primal and dual variables simultaneously. 
%
\end{abstract}

\section{Introduction}





\section{Preliminaries}
\label{sect:Preliminaries}

Let us introduce notation and recall some auxiliary definitions from the first part of our study 
\cite{Dolgopolik_ExAugmLagr} and the theory of Sobolev spaces that will be utilised throughout this article.

\subsection{Exact augmented Lagrangians}
\label{subect:ExactAugmLagr}

Let $X$ and $H$ be real Hilbert spaces, $\langle \cdot, \cdot \rangle$ be the inner product in $X$, $H$ or
$\mathbb{R}^n$, depending on the context, and $|\cdot|$ be the Euclidean norm. Let also functions 
$f, g_i \colon X \to \mathbb{R}$ with $i \in M := \{ 1, \ldots, m \}$ and $F \colon X \to H$ be given functions. 
Denote $g(\cdot) = (g_1(\cdot), \ldots, g_m(\cdot))$. 

The classical Lagrangian for the constrained optimization problem
\[
  \min \enspace f(x) \quad \text{subject to} \quad F(x) = 0, \quad g_i(x) \le 0, \quad i \in M
  \qquad \eqno{(\mathcal{P})}
\]
is defined as 
\[
  L(x, \lambda, \mu) = f(x) + \langle \lambda, F(x) \rangle + \langle \mu, g(x) \rangle, 
  \quad \lambda \in H, \quad \mu \in \mathbb{R}^m.
\]
To define an exact augmented Lagrangian for the problem $(\mathcal{P})$, choose a convex non-decreasing lower
semicontinuous (l.s.c.) function $\phi \colon [0, +\infty) \to [0, + \infty]$ such that $\phi(t) = 0$ if and only if 
$t = 0$ and $\dom \phi \ne \{ 0 \}$. Choose also a continuously differentiable concave function 
$\psi \colon [0, + \infty)^m \to \mathbb{R}$ such that zero is a point of global maximum of $\psi$ and $\psi(0) > 0$ 
(some additional assumptions on these function as well as several particular examples can be found in 
\cite{Dolgopolik_ExAugmLagr}).

For any vectors $y, z \in \mathbb{R}^m$ denote by $\max\{ y, z \} \in \mathbb{R}^m$ the coordinate-wise maximum of the
vectors $y$ and $z$. Introduce the functions
\[
  b(x) = \psi\big( \max\{ g(x), 0 \} \big), \quad p(x, \mu) = \frac{b(x)}{1 + |\mu|^2} 
  \quad \forall x \in X, \: \mu \in \mathbb{R}^m
\]
and denote $\Omega = \{ x \in X \mid b(x) > 0, \: \phi(\| F(x) \|^2) < + \infty \}$. 

Let the functions $f$, $F$, and $g_i$, $i \in M$, be continuously Fr\'{e}chet differentiable, $D F(x)[\cdot]: X \to H$ 
be the Fr\'{e}chet derivative of $F$ at $x$, and $\nabla_x L(x, \lambda, \mu)$ be the gradient of the function 
$x \mapsto L(x, \lambda, \mu)$. The exact augmented Lagrangian for the problem $(\mathcal{P})$ is defined as follows:
\begin{multline} \label{eq:ExactAugmLagrDef}
  \mathscr{L}(x, \lambda, \mu, c) 
  = f(x) + \langle \lambda, F(x) \rangle + \frac{c}{2} \big( 1 + \| \lambda \|^2 \big) \phi(\| F(x) \|^2)
  \\
  + \Big\langle \mu, \max\Big\{ g(x), - \frac{1}{c} p(x, \mu) \mu \Big\} \Big\rangle
  + \frac{c}{2 p(x, \mu)} \Big| \max\Big\{ g(x), - \frac{1}{c} p(x, \mu) \mu \Big\} \Big|^2
  \\
  + \eta(x, \lambda, \mu),
\end{multline}
if $x \in \Omega$, and $\mathscr{L}(x, \lambda, \mu, c) = + \infty$, otherwise. Here $\lambda \in H$ and 
$\mu \in \mathbb{R}^m$ are Lagrange multipliers, $c > 0$ is the penalty parameter, and
\begin{equation} \label{eq:KKTPenTerm}
\begin{split}
  \eta(x, \lambda, \mu) = &\frac{1}{2} \Big\| D F(x)\big[ \nabla_x L(x, \lambda, \mu) \big] \Big\|^2
  \\
  &+ \frac{1}{2} \sum_{i = 1}^m 
  \Big( \langle \nabla g_i(x), \nabla_x L(x, \lambda, \mu) \rangle + g_i(x)^2 \mu_i \Big)^2.
\end{split}
\end{equation}
For the sake of simplicity, below we will consider only the case $\phi(t) \equiv t$ and $\psi(y) \equiv 1$, although
optimization methods based on exact augmented Lagrangians for problems presented in the article might perform better for
a different choice of these functions.

For any $\gamma \in \mathbb{R}$ and $c > 0$ introduce sublevel sets
\begin{align} \notag
  S_c(\gamma) &= \Big\{ (x, \lambda, \mu) \in \Omega \times H \times \mathbb{R}^m 
  \Bigm| \mathscr{L}(x, \lambda, \mu, c) \le \gamma \Big\},
  \\ \label{eq:PenFuncSublevelSet}
  \Omega_c(\gamma) 
  &= \Big\{ x \in \Omega \Bigm| f(x) + c\big( \| F(x) \|^2 + |\max\{ g(x), 0 \}^2| \big) \le \gamma \Big\}.
\end{align}
Let us recall the definition of global exactness of the augmented Lagrangian. Suppose that there exists 
a globally optimal solution of the problem $(\mathcal{P})$. Recall that a triplet 
$(x, \lambda, \mu) \in X \times H \times \mathbb{R}^m$ is called a KKT point of the problem $(\mathcal{P})$, if $x$ is
feasible for this problem, $\nabla_x L(x, \lambda, \mu) = 0$, and for all $i \in M$ one has $\mu_i g_i(x) = 0$ and 
$\mu_i \ge 0$. Suppose that for any globally optimal solution $x_*$ of the problem $(\mathcal{P})$ there exist Lagrange
multipliers $\lambda_* \in H$ and $\mu_* \in \mathbb{R}^m$ such that the triplet $(x_*, \lambda_*, \mu_*)$ is a KKT
point of the problem $(\mathcal{P})$.

\begin{definition}
The augmented Lagrangian $\mathscr{L}(x, \lambda, \mu, c)$ is called \textit{globally exact}, if there exists $c_* > 0$
such that for all $c \ge c_*$ a triplet $(x_*, \lambda_*, \mu_*)$ is a globally optimal solution of the problem
\begin{equation} \label{prob:ExAugmLagr}
  \min_{(x, \lambda, \mu) \in X \times H \times \mathbb{R}^m} \mathscr{L}(x, \lambda, \mu, c)
\end{equation}
if and only if $x_*$ is a globally optimal solution of the problem $(\mathcal{P})$ and $(x_*, \lambda_*, \mu_*)$ is a
KKT point of this problem.
\end{definition}

\begin{remark}
It should be noted that under some natural assumptions the augmented Lagrangian $\mathscr{L}(x, \lambda, \mu, c)$ is
globally exact if and only if there exists $c_*$ such that for any $c \ge c_*$ the optimal value of problem
\eqref{prob:ExAugmLagr} is equal to the optimal value of the problem $(\mathcal{P})$
(see \cite[Lemma~5.2]{Dolgopolik_ExAugmLagr}).
\end{remark}

Sufficient conditions for the global exactness of the augmented Lagrangian $\mathscr{L}(x, \lambda, \mu, c)$ are
expressed in terms of the function
\begin{equation} \label{eq:Q_def}
\begin{split}
  Q(x)[\lambda, \mu] = 
  &\frac{1}{2} \Big\| DF(x) \Big[ DF(x)^*[\lambda] + \sum_{i = 1}^m \mu_i \nabla g_i(x) \Big] \Big\|^2
  \\
  &+ \frac{1}{2} 
  \Big| \nabla g(x) \Big( DF(x)^*[\lambda] + \sum_{i = 1}^m \mu_i \nabla g_i(x) \Big) + \diag(g_i(x)^2) \mu \Big|^2
\end{split}
\end{equation}
where $DF(x)^*[\cdot]$ is the adjoint operator of the linear operator $DF(x)[\cdot]$ and 
$\nabla g(x) y \in \mathbb{R}^m$ is the vector whose $i$-th coordinate is $\langle \nabla g_i(x), y \rangle$ for any 
$y \in X$. Note that the function $Q(x)[\cdot]$ is quadratic in $(\lambda, \mu)$.

\begin{definition} \label{def:Q_PositiveDef}
The function $Q(x)[\cdot]$ is called \textit{positive definite} at a point $x \in X$, if there exists $\sigma > 0$ such
that
\[
  Q(x)[\lambda, \mu] \ge \sigma \big( \|\lambda\|^2 + |\mu|^2 \big) 
  \quad \forall \lambda \in H, \: \mu \in \mathbb{R}^m
\]
The supremum of all those $\sigma \ge 0$ for which the inequality above holds true is denoted by $\sigma_{\max}(Q(x))$.
\end{definition}

Sufficient conditions for the global exactness and the stronger propery of the so-called \textit{complete} exactness of
the augmented Lagrangian $\mathscr{L}(x, \lambda, \mu, c)$ obtained in the first part of our study 
\cite{Dolgopolik_ExAugmLagr} heavily rely on the assumption of weak sequential lower semicontinuous of the
function $\mathscr{L}(\cdot, c)$ and the boundedness of of the sublevel set $S_c(\gamma)$ for a properly chosen value
of $\gamma$. These assumptions are needed to ensure that problem \eqref{prob:ExAugmLagr} has globally optimal solutions
for any sufficiently large value of the penalty parameter $c > 0$. However, as we will see below, in some applications
these assumptions are very difficult (if at all possible) to verify, which renders the results from
\cite{Dolgopolik_ExAugmLagr} inapplicable. To overcome this difficulty, below we present sufficient conditions for the
global and complete exactness of the augmented Lagrangian $\mathscr{L}(x, \lambda, \mu, c)$ that do not rely on lower
semicontinuity or boundedness assumptions. These conditions were inspired by the results of Demyanov
\cite{Demyanov2003,Demyanov2004,Demyanov2005,DemyanovGiannessi2003,DemyanovTamasyan2011,DemyanovTamasyan2014} on exact
penalty functions for optimization problems in infinite dimensional space, in which the assumption on the existence of
global minimizers of the penalty function is stated explicitly.

\begin{theorem} \label{THM:COMPLETE_EXACTNESS}
Let $f_* > - \infty$ be the optimal value of the problem $(\mathcal{P})$ and the following assumptions be valid:
\begin{enumerate}
\item{$f$, $g_i$, $i \in M$, and $F$ are twice continuously Fr\'{e}chet differentiable on $\Omega$,
$\phi$ is continuously differentiable on its effective domain, and $\phi'(0) > 0$;}

\item{there exist $r > 0$, $\varepsilon > 0$ and $\gamma > f_*$ such that the functions $f$, $g_i$, $i \in M$, and $F$,
as well as their first and second order Fr\'{e}chet derivatives, are bounded on $\Omega_r(\gamma + \varepsilon)$;}

\item{there exists $\sigma > 0$ such that $\sigma_{\max}(Q(x)) \ge \sigma$ for all 
$x \in \Omega_r(\gamma + \varepsilon)$;}

\item{there exists an unbounded increasing sequence $\{ c_n \}$ such that for any $n \in \mathbb{N}$ the function
$\mathscr{L}(\cdot, c_n)$ attains a global minimum on $X \times H \times \mathbb{R}^m$.}
\end{enumerate}
Then there exists $c_* > 0$ such that for all $c \ge c_*$ the augmented Lagrangian $\mathscr{L}(x, \lambda, \mu, c)$ is
completely exact on the set $S_c(\gamma)$ in the following sense:
\begin{enumerate}
\item{the optimal values of the problem $(\mathcal{P})$ and problem \eqref{prob:ExAugmLagr} coincide;
}

\item{$(x_*, \lambda_*, \mu_*)$ is point of global minimum of $\mathscr{L}(x, \lambda, \mu, c)$ if and only if 
$x_*$ is a globally optimal solution of the problem $(\mathcal{P})$ and $(x_*, \lambda_*, \mu_*)$ is a KKT-point 
of this problem;
}

\item{$(x_*, \lambda_*, \mu_*) \in S_c(\gamma)$ is a stationary point of $\mathscr{L}(x, \lambda, \mu, c)$ if and only
if $(x_*, \lambda_*, \mu_*)$ is a KKT-point of the problem $(\mathcal{P})$ and $f(x_*) \le \gamma$;
}

\item{if $(x_*, \lambda_*, \mu_*) \in S_c(\gamma)$ is a point of local minimum of $\mathscr{L}(x, \lambda, \mu, c)$, 
then $x_*$ is a locally optimal solution of the problem $(\mathcal{P})$, $f(x_*) \le \gamma$, and 
$(x_*, \lambda_*, \mu_*)$ is a KKT-point of this problem.
}
\end{enumerate}
\end{theorem}

The proof of this theorem is given in the appendix.

\begin{remark} \label{rmrk:BasicExistenceAssumption}
Throughout this article we implicitly assume that all optimization problems under consideration have finite optimal
value and a globally optimal solution (except for the problems involving the augmented Lagrangian). Furthermore, we
also suppose that for any globally optimal solution of these problems there exist Lagrange multiplier, for which the
KKT optimality conditions hold true.
\end{remark}

\subsection{$L^p$ and Sobolev spaces}
\label{subsect:SobolevSpaces}

Let $\mathscr{U}$ be a Hilbert space and $E \subset \mathbb{R}^n$ be an open set. For any $1 \le p \le \infty$ we denote
by $L^p(E; \mathscr{U})$ the space consisting of all those measurable functions $u \colon E \to \mathscr{U}$, for which 
\begin{align*}
  \| u \|_{L^p(E; \mathscr{U})} 
  &:= \left( \int_E \| u(x) \|_{\mathscr{U}}^p \, dx \right)^{\frac{1}{p}} < + \infty, 
  \quad \text{ if } 1 \le p < \infty, 
  \\
  \| u \|_{L^{\infty}(E; \mathscr{U})} &:= \esssup_{x \in E} \| u(x) \|_{\mathscr{U}} < + \infty,
\end{align*}
endowed with the norm $\| \cdot \|_{L^p(E; \mathscr{U})}$. If $\mathscr{U}$ is the space $\mathbb{R}^d$ equipped
with the Euclidean norm $|\cdot|$, then we denote the corresponding norm on $L^p(E; \mathbb{R}^d)$ by $\| \cdot \|_p$.

Let $W^{s, p}(E)$ be the Sobolev space, while $W^{s, p}(E; \mathbb{R}^d)$ be the space consisting of all
those vector-valued functions $u \colon E \to \mathbb{R}^d$, $u = (u_1, \ldots, u_d)$, for which 
$u_i \in W^{s, p}(E)$ for all $i \in \{1, \ldots, d\}$. The space $W^{1, 2}(E; \mathbb{R}^d)$ is equipped with
the inner product
\[
  \langle u, v \rangle 
  = \int_E \big( \langle u(x), v(x) \rangle + \langle \nabla u(x), \nabla v(x) \rangle \big) \, dx 
  \quad \forall u, v \in W^{1, 2}(E; \mathbb{R}^d).
\]
and the corresponding norm. The closure of the space $C^{\infty}_0(E; \mathbb{R}^d)$ of all infinitely
continuously differentiable functions $u \colon E \to \mathbb{R}^d$ with compact support in the space 
$W^{1, 2}(E; \mathbb{R}^d)$ is denoted by $W^{1, 2}_0(E; \mathbb{R}^d)$.

As is well known (see, e.g. \cite{Leoni}), in the case when $E = (a, b) \subset \mathbb{R}$ the space 
$W^{1, 2}((a, b); \mathbb{R}^d)$ is isomorphic to the space of all absolutely continuous functions 
$u \colon [a, b] \to \mathbb{R}^d$ such that $u' \in L^2((a, b); \mathbb{R}^d)$, which we denote by 
$W^{1, 2}([a, b]; \mathbb{R}^d)$.

We will also use a Sobolev-like function space introduced in \cite{Dolgopolik_MultidimCalcVar}. Let us recall 
the definition of this space. Any $n$-tuple $\alpha = (\alpha_1, \ldots, \alpha_n) \in \mathbb{Z}_+^n$ of nonnegative
integers $\alpha_i$ is called a \textit{multi-index}. Its absolute value is defined as 
$|\alpha| = \alpha_1 + \ldots + \alpha_n$. For any multi-index $\alpha$ denote by 
$D^{\alpha} = D_1^{\alpha_1} \ldots D^{\alpha_n}_n$ a differential operator of order $|\alpha|$, where 
$D_i = \partial/\partial x_i$. If $\alpha = 0$, then $D^{\alpha} u = u$ for any function $u$. 
For any $k \in \{ 0, \ldots, n \}$ define
\[
  I_k = \Big\{ \alpha \in \mathbb{Z}^n_+ \Bigm| |\alpha| = k, \enspace \alpha_i = 0 \text{ or } \alpha_i = 1
  \enspace \forall i \in \{ 1, \ldots, n \} \Big\}.
\]
Finally, denote by $MW^{n, 2}(E)$ the set of all function $u \in L^2(E)$ such that for any 
$k \in \{ 1, \ldots, n \}$ and $\alpha \in I_k$ there exists the weak derivative $D^{\alpha} u$ of $u$ that belongs to
$L^2(E)$. Thus, $MW^{n, 2}$ consists of all those functions $u \in L^2(E)$ for which there exist all weak
mixed derivatives of the order $k \in \{ 1, \ldots, n \}$ that belong to $L^2(E)$. Note that 
$W^{n, 2}(E) \subseteq MW^{n, 2}(E)$ and this inclusion turns into an equality in the case $n = 1$.

The linear space $MW^{n, 2}(E)$ equipped with the following inner product 
\[
  \langle u, v \rangle = \sum_{k = 0}^n \sum_{\alpha \in I_k} \int_E D^{\alpha} u(x) D^{\alpha} v(x) \, dx
  \quad \forall u, v \in MW^{n, 2}(E)
\]
and the corresponding norm $\| \cdot; MW^{n, 2} \|$ is a separable Hilbert space. The closure of $C_0^{\infty}(E)$
in $MW^{n, 2}(E)$ is denoted by $MW_0^{n, 2}(E)$. By \cite[Thm.~2]{Dolgopolik_MultidimCalcVar} the seminorm on
$MW_0^{n, 2}(E)$ corresponding to the inner product
\[
  \langle u, v \rangle = \int_E D^{(1, \ldots, 1)} u(x) D^{(1, \ldots, 1)} v(x) \, dx
  \quad \forall u, v \in MW_0^{n, 2}(E)
\]
is a norm that is equivalent to the norm $\| \cdot; MW^{n, 2} \|$. Therefore, below we suppose that the space 
$MW_0^{n, 2}(E)$ is endowed with this inner product and the corresponding norm.

Functions from the space $MW_0^{n, 2}(E)$ admit a very simple and convenient characterisation in the case when
$E = \prod_{i = 1}^n (a_i, b_i)$ is a bounded open box in $\mathbb{R}^n$. In this case for any $v \in L^2(E)$ denote
\begin{equation} \label{eq:MultidimIntegralOper}
  (\mathcal{A} v)(x) = \int_{a_1}^{x_1} \ldots \int_{a_n}^{x_n} v(s) \, ds \quad \text{for a.e. } x \in E.
\end{equation}
By the Fubini theorem the function $\mathcal{A} v$ is correctly defined and $\mathcal{A}$ is a continuous linear
operator mapping $L^2(E)$ to $L^2(E)$. The following result holds true (see
\cite[Thm.~3]{Dolgopolik_MultidimCalcVar}).

\begin{theorem} \label{thm:MixedSobolev}
Let $E = \prod_{i = 1}^n (a_i, b_i)$. Then a function $u \colon E \to \mathbb{R}$ belongs to $MW^{n, 2}_0(E)$ if and
only if there exist a function $v \in L^2(E)$ such that:
\begin{enumerate}
\item{$u(x) = (\mathcal{A} v)(x)$ for a.e. $x \in E$;}

\item{$\int_{a_i}^{b_i} v(x_1, \ldots, x_{i - 1}, s_i, x_{i + 1}, \ldots, x_n) \, ds_i = 0$ for a.e. $x \in E$ and
for all $i \in \{ 1, \ldots, n \}$.}
\end{enumerate}
Moreover, such function $v$ is uniquely defined and $v = D^{(1, \ldots, 1)} u$ in the weak sense.
\end{theorem}

\section{Applications to the calculus of variations}
\label{sect:CalculusOfVariations}

In this section, we consider applications of the theory of exact augmented Lagrangians to constrained problems of the
calculus of variations. Namely, we present new exact augmented Lagrangians for variational problems with additional
constraints at the boundary, problems with isoperimetric constraints, and variational problems with nonholonomic
equality constraints. We provide sufficient conditions on the problem data ensuring global/complete exactness of these
augmented Lagrangians. For the sake of simplicity we consider each type of constraints separately, although one can
define an exact augmented Lagrangian for variational problems involving all types of the aforementioned constraints
and with the use of the results presented below derive sufficient conditions for the global/complete exactness of such
augmented Lagrangian.

\subsection{Problems with additional constraints at the boundary}
\label{subsect:BoundaryConstraints}

Consider the following variational problems with additional nonlinear equality and inequality constraints at
the boundary of the domain $(a, b) \subset \mathbb{R}$:
\begin{equation} \label{prob:BoundaryConstraints}
\begin{split}
  &\min_{u \in W^{1, 2}([a, b]; \mathbb{R}^d)} \enspace 
  \mathcal{I}(u) = \int_a^b f(u(x), u'(x), x) \, dx + f_0(u(a), u(b))
  \\
  &\text{subject to} \enspace f_j(u(a), u(b)) = 0, \enspace j \in J, \enspace g_i(u(a), u(b)) \le 0, \enspace i \in M.
\end{split}
\end{equation}
Here $f \colon \mathbb{R}^d \times \mathbb{R}^d \times [a, b] \to \mathbb{R}$, $f = f(u, \xi, x)$, 
$f_j \colon \mathbb{R}^{2d} \to \mathbb{R}$, and $g_i \colon \mathbb{R}^{2d} \to \mathbb{R}$ are given function, and 
$J = \{ 1, \ldots, \ell \}$ and $M = \{ 1, \ldots, m \}$ are finite index sets any one of which can be empty. We suppose
that the functions $f_j$ and $g_i$ are continuously differentiable, while $f$ is a Carath\'{e}odoary function (i.e. the
function $f(u, \xi, \cdot)$ is measurable for all $u, \xi \in \mathbb{R}^d$ and the function $f(\cdot, x)$ is continuous
for a.e. $x \in [a, b])$ that is differentiable in $u$ and $\xi$ and the gradients $\nabla_u f$ and $\nabla_{\xi} f$ are
Carath\'{e}odory functions as well.

We also impose the following growth conditions on $f$ and its partial derivatives. Nanely, for any $R > 0$ there exist
$C_R > 0$ and a.e. nonnegative functions $\eta_R \in L^1(a, b)$ and $\theta_R \in L^2(a, b)$ such that
\begin{equation} \label{eq:1dimGrowthCond}
\begin{split}
  &|f(u, \xi, x)| \le C_R |\xi|^2 + \eta_R(x), \quad |\nabla_u f(u, \xi, x)| \le C_R |\xi|^2 + \eta_R(x), \quad
  \\
  & |\nabla_{\xi} f(u, \xi, x)| \le C_R |\xi| + \theta_R(x).
\end{split}
\end{equation}
for a.e. $x \in [a, b]$ and any $u, \xi \in \mathbb{R}^d$ such that $|u| \le R$.  These growth conditions ensure that
for any $u \in W^{1, 2}([a, b]; \mathbb{R}^d)$ the value $\mathcal{I}(u)$ is correctly defined and finite, the
functional $\mathcal{I}$ is G\^{a}teaux differentiable and its G\^{a}teaux derivative $\mathcal{I}'(u)[\cdot]$ has the
form
\begin{multline} \label{eq:1dimVarFuncDeriv}
  \mathcal{I}'(u)[w] = \int_a^b \big( \langle \nabla_u f(u(x), u'(x), x), w(x) \rangle 
  + \langle \nabla_{\xi} f(u(x), u'(x), x), w'(x) \rangle \big) \, dx 
  \\
  + \langle \nabla f_0(u(a), u(b)), (w(a), w(b)) \rangle
\end{multline}
for all $w \in W^{1, 2}([a, b]; \mathbb{R}^d)$ (see, e.g. \cite[Section~4.3]{Dacorogna}). Furthermore, it is easily seen
that
\begin{align*}
  \big| \mathcal{I}'(u)[w] - \mathcal{I}'(v)[w] \big| 
  &\le \big\| \nabla_u f(u(\cdot), u'(\cdot), \cdot) - \nabla_u f(v(\cdot), v'(\cdot), \cdot) \big\|_1 \| w \|_{\infty}
  \\
  &+ \big\| \nabla_{\xi} f(u(\cdot), u'(\cdot), \cdot) - \nabla_{\xi} f(v(\cdot), v'(\cdot), \cdot) \big\|_2 \| w' \|_2
  \\
  &+ \big| \nabla f_0(u(a), u(b)) - \nabla f_0(v(a), v(b)) \big| \| w \|_{\infty}
\end{align*}
for any $u, v, w \in W^{1, 2}([a, b]; \mathbb{R}^d)$. Hence bearing in mind the fact that 
$\| \cdot \|_{\infty} \le C \| \cdot \|_{1, 2}$ for some $C > 0$ (this result follows from the Sobolev imbedding theorem
\cite[Thm.~5.4]{Adams}) one gets that
\begin{align*}
  \| \mathcal{I}'(u) - \mathcal{I}'(v) \| 
  &\le C \big\| \nabla_u f(u(\cdot), u'(\cdot), \cdot) - \nabla_u f(v(\cdot), v'(\cdot), \cdot) \big\|_1
  \\
  &+ \big\| \nabla_{\xi} f(u(\cdot), u'(\cdot), \cdot) - \nabla_{\xi} f(v(\cdot), v'(\cdot), \cdot) \big\|_2
  \\
  &+ C \big| \nabla f_0(u(a), u(b)) - \nabla f_0(v(a), v(b)) \big|.
\end{align*}
With the use of this inequality and standard results on the continuity of Nemytskii operators (see, e.g.
\cite{AppellZabrejko}) one can conclude that the G\^{a}teaux derivative of $\mathcal{I}$ is continuous, which implies
that this functional is continuously Fr\'{e}chet differentiable.

In order to define an exact augmented Lagrangian for problem \eqref{prob:BoundaryConstraints}, we need to compute the
gradient of the objective function (the functional $\mathcal{I}$) and all constraints of this problem in the space
$W^{1, 2}([a, b], \mathbb{R}^d)$. To simplify this problem, we will use a trick (a change of variables) called
\textit{transition into the space of derivatives}, that was widely utilised by Demyanov in his works 
on the calculus of variations
\cite{Demyanov2003,Demyanov2004,Demyanov2005,DemyanovGiannessi2003,DemyanovTamasyan2011,DemyanovTamasyan2014}.

For any $y \in \mathbb{R}^d$ and $v \in L^2((a, b); \mathbb{R}^d)$ denote 
$\mathcal{A}(y, v)(x) = y + \int_a^x v(s) \, ds$ for all $x \in [a, b]$. The linear operator $\mathcal{A}$ continuously
maps $\mathbb{R}^d \times L^2((a, b); \mathbb{R}^d)$ to $W^{1, 2}([a, b]; \mathbb{R}^d)$. Furthermore, the Lebesgue
differentiation theorem implies that $\mathcal{A}$ is a one-to-one correspondence between these spaces with the inverse
operator $\mathcal{A}^{-1}(u) = (u(a), u'(\cdot))$ (see, e.g. \cite{Leoni}). Therefore, by applying the change of
variables $u = \mathcal{A}(y, v)$ we can convert problem \eqref{prob:BoundaryConstraints} into an equivalent variational
problem of the form:
\begin{equation} \label{prob:BoundaryConstr_DerivSpace}
\begin{split}
  &\min_{(y, v) \in \mathbb{R}^d \times L^2((a, b); \mathbb{R}^d)} \: 
  \widehat{\mathcal{I}}(y, v) := \mathcal{I}(\mathcal{A}(y, v))
  \\
  &\text{subject to} \enspace \widehat{f}_j(y, v) = 0, \enspace j \in J, 
  \enspace  \widehat{g}_i(y, v) \le 0, \enspace i \in M.
\end{split}
\end{equation}
where $\widehat{f}_j(y, v) := f_j(y, y + \int_a^b v(s) \, ds)$ and
$\widehat{g}_i(y, v) := g_i(y, y + \int_a^b v(s) \, ds)$. We will define an exact augmented Lagrangian for 
problem \eqref{prob:BoundaryConstr_DerivSpace} formulated in ``the space of derivatives''. 

Denote $X = \mathbb{R}^d \times L^2((a, b); \mathbb{R}^d)$ and endow the space $X$ with the inner product
\[
  \langle (y, v), (z, w) \rangle = \langle y, z \rangle + \int_a^b \langle v(x), w(x) \rangle \, dx
  \quad \forall (y, v), (z, w) \in X
\]
and the corresponding norm. Fix some $(y, v) \in X$ and for the sake of convenience denote $u = \mathcal{A}(y, v)$. From
equality \eqref{eq:1dimVarFuncDeriv} it follows that the functional $\widehat{\mathcal{I}}$ is continuously Fr\'{e}chet
differentiable and its Fr\'{e}chet derivative $D \widehat{\mathcal{I}}(y, v)[\cdot]$ has the form
\begin{multline*}
  D \widehat{\mathcal{I}}(y, v)[z, w] 
  \\
  = \int_a^b \Big( \Big\langle \nabla_u f(u(x), v(x), x), z + \int_a^x w(s) \, ds \Big\rangle 
  + \langle \nabla_{\xi} f(u(x), v(x), x), w(x) \rangle \Big) \, dx
  \\
  + \langle \nabla_{u_a} f_0(y, u(b)) + \nabla_{u_b} f_0(y, u(b)), z \rangle 
  + \Big\langle \nabla_{u_b} f_0(y, u(b)), \int_a^b v(x) \, dx \Big\rangle
\end{multline*}
for any $(z, w) \in X$, where $f_0 = f_0(u_a, u_b)$. Integrating by parts and rearranging the terms one gets
\begin{multline*}
  D \widehat{\mathcal{I}}(y, v)[z, w] 
  \\
  = \Big\langle \int_a^b \nabla_u f(u(x), v(x), x) \, dx + \nabla_{u_a} f_0(y, u(b)) + \nabla_{u_b} f_0(y, u(b)), 
  z \Big\rangle
  \\
  + \int_a^b \langle P[y, v](x), w(x) \rangle \, dx
\end{multline*}
for all $(z, w) \in X$, where
\[
  P[y, v](x) = \int_x^b \nabla_u f(u(s), v(s), s) \, d s + \nabla_{\xi} f(u(x), v(x), x) 
  + \nabla_{u_b} f_0(y, u(b))
\]
for a.e. $x \in (a, b)$. With the use of the growth conditions \eqref{eq:1dimGrowthCond} one can readily check that 
$P[y, v] \in L^2((a, b); \mathbb{R}^d)$. Therefore, one can conclude that the gradient 
$\nabla \widehat{\mathcal{I}}(y, v) \in X$ of the functional $\widehat{\mathcal{I}}$ has the form
\begin{multline} \label{eq:1dimVarFuncGrad}
  \nabla \widehat{\mathcal{I}}(y, v) 
  \\
  = \Big( \int_a^b \nabla_u f(u(x), v(x), x) \, dx + \nabla_{u_a} f_0(y, u(b)) + \nabla_{u_b} f_0(y, u(b)), 
  P[y, v] \Big)
\end{multline}
for all $(y, v) \in \mathbb{R}^d \times L^2((a, b); \mathbb{R}^d$. Arguing in a similar way one can readily verify that
\begin{equation} \label{eq:BoundaryConstrGrad}
\begin{split}
  \nabla \widehat{f}_j(y, v) 
  &= \Big( \nabla_{u_a} f_j(y, u(b)) + \nabla_{u_b} f_j(y, u(b)), \nabla_{u_b} f_j(y, u(b)) \Big) \in X,
  \\
  \nabla \widehat{g}_i(y, v) 
  &= \Big( \nabla_{u_a} g_i(y, u(b)) + \nabla_{u_b} g_i(y, u(b)), \nabla_{u_b} g_i(y, u(b)) \Big) \in X
\end{split}
\end{equation}
for all $j \in J$ and $i \in M$, where $f_j = f_j(u_a, u_b)$ and $g_i = g_i(u_a, u_b)$.

\begin{remark}
It should be pointed out that with the use of Demyanov's ``transition into the space of derivatives'' technique we were
able to obtain simple and explicit analytical expressions for the gradients of the objective functional and constraints
in a straightforward manner. Without this technique, computation of the gradients of these functions is a challenging
task. 
\end{remark}

Denote $\widehat{F} = (\widehat{f}_1, \ldots, \widehat{f}_{\ell})$ and 
$\widehat{G} = (\widehat{g}_1, \ldots, \widehat{g}_m)$. The classical Lagrangian for problem
\eqref{prob:BoundaryConstr_DerivSpace} has the form
\[
  L(y, v, \lambda, \mu) = \widehat{\mathcal{I}}(y, v) 
  + \langle \lambda, \widehat{F}(y, v) \rangle + \langle \mu, \widehat{G}(y, v) \rangle
  \quad \forall \lambda \in \mathbb{R}^{\ell}, \: \mu \in \mathbb{R}^m.
\]
Its gradient $\nabla_{(u, v)} L(y, v, \lambda, \mu)$ in $(y, v)$ can be easily computed with the use of
\eqref{eq:1dimVarFuncGrad} and \eqref{eq:BoundaryConstrGrad}. An exact augmented Lagrangian 
$\mathscr{L}(y, v, \lambda, \mu, c)$ with $\lambda \in \mathbb{R}^{\ell}$ and $\mu \in \mathbb{R}^m$ for problem 
\eqref{prob:BoundaryConstr_DerivSpace} is defined according to equalities \eqref{eq:ExactAugmLagrDef} and 
\eqref{eq:KKTPenTerm}. For the sake of brevity and convenience, we will write in explicitly in terms of the original
problem \eqref{prob:BoundaryConstraints}:
\begin{align*}
  \mathscr{L}(u, \lambda, \mu, c) &= \int_a^b f(u(x), u'(x), x) \, dx + f_0(u(a), u(b))
  \\
  &+ \langle \lambda, F(u(a), u(b)) \rangle + \frac{c}{2} (1 + |\lambda|^2) |F(u(a), u(b))|^2
  \\
  &+ \Big\langle \mu, \max\Big\{ G(u(a), u(b)), - \frac{1}{c(1 + |\mu|^2)} \mu \Big\} \Big\rangle
  \\
  &+ \frac{c}{2} (1 + |\mu|^2) \Big|\max\Big\{ G(u(a), u(b)), - \frac{1}{c(1 + |\mu|^2)} \mu \Big\}\Big|^2
  + \eta(u, \lambda, \mu),
\end{align*}
where $F = (f_1, \ldots, f_{\ell})$, $G = (g_1, \ldots, g_m)$, and
\begin{align*}
  \eta(u, \lambda, \mu) &= \frac{1}{2} \sum_{j = 1}^{\ell}
  \Big( \langle \nabla_{u_a} f_j(u(a), u(b)) + \nabla_{u_b} f_j(u(a), u(b)), \nabla_y L(u, \lambda, \mu) \rangle
  \\
  &+ \int_a^b \langle \nabla_{u_b} f_j(u(a), u(b)), \nabla_v L(u, \lambda, \mu)(x) \rangle \, dx \Big)^2
  \\
  &+ \frac{1}{2} \sum_{i = 1}^{m}
  \Big( \langle \nabla_{u_a} g_i(u(a), u(b)) + \nabla_{u_b} g_i(u(a), u(b)), \nabla_y L(u, \lambda, \mu) \rangle 
  \\
  &+ \int_a^b \langle \nabla_{u_b} g_i(u(a), u(b)), \nabla_v L(u, \lambda, \mu)(x) \rangle \, dx 
  + g_i(u(a), u(b)^2 \mu_i \Big)^2,
\end{align*}
and
\begin{align*}
  \nabla_y L(u, \lambda, \mu) &= \int_a^b \nabla_u f(u(x), u'(x), x) \, dx 
  + \nabla_{u_a} f_0(u(a), u(b)) + \nabla_{u_b} f_0(u(a), u(b))
  \\
  &+ \sum_{j = 1}^{\ell} \lambda_j \big( \nabla_{u_a} f_j(u(a), u(b)) + \nabla_{u_b} f_j(u(a), u(b)) \big)
  \\
  &+ \sum_{i = 1}^m \mu_i \big( \nabla_{u_a} g_i(u(a), u(b)) + \nabla_{u_b} g_i(u(a), u(b)) \big),
\end{align*}
and
\[
  \nabla_v L(u, \lambda, \mu)(x) = P[u](x) + \sum_{j = 1}^{\ell} \lambda_j \nabla_{u_b} f_j(u(a), u(b))
  + \sum_{i = 1}^m \mu_i \nabla_{u_b} g_i(u(a), u(b)),
\]
and 
\[
  P[u](x) = \int_x^b \nabla_u f_0(u(s), u'(s), s) \, ds + \nabla_{\xi} f_0(u(x), u'(x), x)
  + \nabla_{u_b} f_0(u(a), u(b))
\]
for a.e. $x \in (a, b)$. 

Let us provide sufficient conditions for the exactness of the augmented Lagrangian $\mathscr{L}(y, v, \lambda, \mu, c)$,
which ensure that the problem
\begin{equation} \label{prob:ExAugmLagr_BoundaryConstr}
  \min_{(y, v, \lambda, \mu) \in \mathbb{R}^d \times L^2((a, b); \mathbb{R}^d) \times \mathbb{R}^{\ell + m}} 
  \mathscr{L}(y, \mu, \lambda, \mu, c)
\end{equation}
is (in some sense) equivalent to problem \eqref{prob:BoundaryConstraints} for any sufficiently large $c$. We will
formulate such conditions in terms of the functions $f_0$, $f_j$, $j \in J$, $g_i$, $i \in M$, and their derivatives.
To avoid cumbersome assumptions on the function $f_0$ and its first and second order derivatives, we will assume that
$f_0$ is quadratic in $\xi$ and bounded below in $u$, although the following theorem can be proved under less
restrictive assumptions on the function $f_0$.

To conveniently formulate the sufficient conditions, recall that a function $\varphi \colon X \to \mathbb{R}$ is called
\textit{coercive}, if $\varphi(x_n) \to + \infty$ for any sequence $\{ x_n \} \subseteq X$ such that 
$\| x_n \| \to + \infty$ as $n \to \infty$.

\begin{theorem}
Let the following assumptions be valid:
\begin{enumerate}
\item{the function $f$ has the form
\[
  f(u, \xi, x) = \langle \xi, h_2(x) \xi \rangle + \langle h_1(u, x), \xi \rangle + h_0(u, x)
  \quad \forall u, \xi \in \mathbb{R}^d, \: x \in [a, b]
\]
for some function $h_2 \in L^{\infty}((a, b); \mathbb{R}^{d \times d})$ and some continuous functions 
$h_1 \colon \mathbb{R}^d \times [a, b] \to \mathbb{R}^d$, and $h_0 \colon \mathbb{R}^d \times [a, b] \to \mathbb{R}$
that are twice continuously differentiable in $u$ and such that
\[
  \langle \xi, h_2(x) \xi \rangle \ge \varkappa |\xi|^2 
  \quad \forall \xi \in \mathbb{R}^d, \: x \in [a, b]
\]
for some $\varkappa > 0$, and $|h_1(u, x)| \le \eta(x)$ and $h_0(u, x) \ge -\eta(x)$ 
for some a.e. nonnegative function $\eta \in L^2(a, b)$ and all $u \in \mathbb{R}^d$ and $x \in [a, b]$;
}

\item{the functions $f_j$, $j \in J \cup \{ 0 \}$, and $g_i$, $i \in M$, are twice continuously differentiable and 
the function $f_0(\cdot) + c(|F(\cdot)|^2 + |\max\{ G(\cdot), 0 \}|^2)$ is coercive for some $c > 0$;}

\item{the gradients $\nabla f_j(z)$, $j \in J$, and $\nabla g_i(z)$, $i \in \{ k \in M \mid g_k(z) = 0 \}$ are linearly
independent for any $z \in \mathbb{R}^{2d}$.}
\end{enumerate}
Then the augmented Lagrangian $\mathscr{L}(y, v, \lambda, \mu, c)$ is globally exact and for any $\gamma \in \mathbb{R}$
there exists $c(\gamma) > 0$ such that for all $c \ge c(\gamma)$ the following statements hold true:
\begin{enumerate}
\item{the sublevel set $S_c(\gamma)$ is bounded;}

\item{if $(y_*, v_*, \lambda_*, \mu_*) \in S_c(\gamma)$ is a point of local minimum of problem
\eqref{prob:ExAugmLagr_BoundaryConstr}, then $u_* = \mathcal{A}(y_*, v_*)$ is a locally optimal solution of
problem \eqref{prob:BoundaryConstraints};
}

\item{a quadruplet $(y_*, v_*, \lambda_*, \mu_*) \in S_c(\gamma)$ is a stationary point of problem
\eqref{prob:ExAugmLagr_BoundaryConstr} if and only if $(u_*, \lambda_*, \mu_*)$ with $u_* = \mathcal{A}(y_*, v_*)$ is a
KKT point of problem \eqref{prob:BoundaryConstraints} with $\mathcal{I}(u_*) \le \gamma$.
}
\end{enumerate}
\end{theorem}

\begin{proof}
Note that the first assumption of the theorem implies that the growth conditions \eqref{eq:1dimGrowthCond} hold true.
Furthermore, one can readily check that the first assumptions also ensures that the functional $\widehat{\mathcal{I}}$
is twice continuously Fr\'{e}chet differentiable and all its first and second order derivatives are bounded on bounded
subsets of the space $X$.

In turn, by \cite[Thm.~3.23]{Dacorogna} the conditions on the functions $h_2$, $h_1$, and $h_0$ guarantee that 
the functional $\mathcal{I}$ (and hence the functional $\widehat{\mathcal{I}}$ as well) is weakly sequentially l.s.c.,
while the mappings
\[
  (y, v) \mapsto \nabla_u f(\mathcal{A}(y, v)(\cdot), v(\cdot), \cdot), \quad
  (y, v) \mapsto \nabla_{\xi} f(\mathcal{A}(y, v)(\cdot), v(\cdot), \cdot)
\]
are weakly sequentially continuous, since they are affine in $v$ (see \cite{Dacorogna_WeakCont}). Note also that if a
sequence $\{ (y_n, v_n) \}$ weakly converges in $X$, then the corresponding sequence $\{ (y_n, u_n(b)) \}$ converges in
$\mathbb{R}^{2d}$. With the use of these facts one can readily check that the augmented Lagrangian 
$\mathscr{L}(y, v, \lambda, \mu, c)$ is weakly sequentially lower semicontinuous in $(y, v)$.

In addition, the first assumption of the theorem implies that
\[
  f(u, \xi, x) \ge \varkappa |\xi|^2 - \eta(x) |\xi| - \eta(x) 
  \quad \forall \xi, u \in \mathbb{R}^d, \text{ a.e. } x \in (a, b)
\]
Therefore, if a sequence $\{ (y_n, v_n) \} \subset X$ is such that $\| v_n \|_2 \to + \infty$ as $n \to \infty$, then
$\int_a^b f(u_n(x), v_n(x), x) \, dx \to + \infty$ as $n \to \infty$. In turn, if the sequence $\{ v_n \}$ is bounded in
$L^2((a, b); \mathbb{R}^d)$, but $|y_n| \to + \infty$ as $n \to \infty$, then the sequence 
$\{ \int_a^b f(u_n(x), v_n(x), x) \, dx \}$ is bounded below,
while 
\[
  f_0(y_n, u_n(b)) + c \Big( |F(y_n, u_n(b))|^2 + |\max\{ G(y_n, u_n(b)), 0 \}|^2 \Big) \to + \infty
\]
as $n \to \infty$ by the second assumption of the theorem. Therefore, the set $\Omega_c(\gamma)$ (see
\eqref{eq:PenFuncSublevelSet}) is bounded for any $c > 0$ and $\gamma \in \mathbb{R}$.

Let us finally check that the third assumption of the theorem guarantees that for any bounded set $V \subset X$ there
exists $\sigma > 0$ such that $\sigma_{\max}(Q(y, v)) \ge \sigma$ for all $(y, v) \in V$. Then applying 
\cite[Cor.~3.6 and Thms.~5.3 and 5.7]{Dolgopolik_ExAugmLagr} one obtains the required result.

First, we show that $Q(y, v)[\cdot]$ is positive definite for any $(y, v) \in X$. Indeed, fix any $(y, v) \in X$. By
\cite[Lemma~3.3]{Dolgopolik_ExAugmLagr} the function $Q(y, v)[\cdot]$ is positive definite if and only if the linear
operator $\mathcal{T} \colon X \to \mathbb{R}^{\ell + m(y, v)}$ defined as
\[
  \mathcal{T}(z, w) = \prod_{j = 1}^{\ell} \big\{ \langle \nabla \widehat{f}_j(y, v), (z, w) \rangle \big\}
  \times \prod_{i \in M(y, v)} \big\{ \langle \nabla \widehat{g}_i(y, v), (z, w) \rangle \big\}
\]
is surjective, where $M(y, v) = \{ i \in M \mid \widehat{g}_i(y, v) = 0 \}$ and $m(y, v)$ is the cardinality of the set
$M(y, v)$.

As is easily seen, the linear operator $\mathcal{T}$ is surjective if and only if the vectors 
\begin{equation} \label{eq:BoundaryConstrGrad_LinIndep}
  \nabla \widehat{f}_j(y, v), \quad j \in \{ 1, \ldots, \ell \}, \quad
  \nabla \widehat{g}_i(y, v), \quad i \in M(y, v),
\end{equation}
are linearly independent. Therefore, the function $Q(y, v)[\cdot]$ is positive definite, provided these vectors are
linearly independent. Let us check that these vectors are indeed linearly independent under the assumptions of the
theorem.

Suppose by contradiction that the vectors \eqref{eq:BoundaryConstrGrad_LinIndep} are linearly dependent. Then bearing 
in mind \eqref{eq:BoundaryConstrGrad} one gets that there exist real numbers $\alpha_j$, $j \in \{ 1, \ldots, \ell \}$,
and $\beta_i$, $i \in M(y, v)$, not all zero, such that
\begin{align*}
  &\sum_{j = 1}^{\ell} \alpha_j (\nabla_{u_a} + \nabla_{u_b}) f_j(y, u(b)) 
  + \sum_{i \in M(y, v)} \beta_j (\nabla_{u_a} + \nabla_{u_b}) g_i(y, u(b)) = 0,
  \\
  &\sum_{j = 1}^{\ell} \alpha_j \nabla_{u_b} f_j(y, u(b)) + \sum_{i \in M(y, v)} \beta_j \nabla_{u_b} g_i(y, u(b)) = 0,
\end{align*}
where $u = \mathcal{A}(y, v)$. Consequently, one has
\[
  \sum_{j = 1}^{\ell} \alpha_j \nabla f_j(y, u(b)) + \sum_{i \in M(y, v)} \beta_j \nabla g_i(y, u(b)) = 0,
\]
which contradicts the third assumption of the theorem. Thus, the function $Q(y, v)[\cdot]$ is positive definite for any
$(y, v) \in X$.

Observe that from the definition of the function $Q(\cdot)$ (see \eqref{eq:Q_def}) it follows that in the case of the
problem under consideration it has the form
\[
  Q(y, v)[\lambda, \mu] 
  = \langle \left( \begin{smallmatrix} \lambda \\ \mu \end{smallmatrix} \right), 
  E_Q(y, u(b)) \left( \begin{smallmatrix} \lambda \\ \mu \end{smallmatrix} \right) \rangle
  \quad \forall \lambda \in \mathbb{R}^{\ell}, \: \mu \in \mathbb{R}^m
\]
for some matrix $E_Q(y, u(b))$ defined via the function $G$ and the partial derivatives of $F$ and $G$ and,
therefore, continuously depending on $(y, u(b))$. 

It is easily seen that $\sigma_{\max}(Q(y, v))$ coincides with the smallest eigenvalue of the matrix $E_Q(y, u(b))$
that continuously depends on $(y, u(b))$, since the matrix $E_Q(y, u(b))$ continuously depends on $(y, u(b))$.
Hence taking into account the fact that the matrix $E_Q(\cdot)$ is positive definite (since for all $(y, v) \in X$
one has $\sigma_{max}(Q(y, v)) > 0$) one can conclude that for any set $V \subset X$ such that the set 
$V(a, b) = \{ (y, u(b)) \in \mathbb{R}^{2d} \mid (y, v) \in V \}$ is bounded, there exists $\sigma > 0$ such that
$\sigma_{\max}(Q(y, v)) \ge \sigma$. It remains to note that if a set $V \subset X$ is bounded, then the set
$V(a, b)$ is bounded as well.
\end{proof}

%

\subsection{Isoperimetric problems}
\label{subsect:Isoperimetric}

Let us now consider an exact augmented Lagrangian for multidimensional variational problems with isoperimetric
constraints. For the sake of simplicity, we restrain our consideration to the case of a single equality isoperimetric
constraint, although all results of this subsection can be extended to the case of variational problems with multiple
equality and inequality isoperimetric constraints.

Let $E = \prod_{i = 1}^n (a_i, b_i) \subset \mathbb{R}^n$ be a bounded open box and 
$\overline{u} \in W^{1, 2}(E; \mathbb{R}^d)$ be a given function. Consider the following variational problem with
an isoperimetic equality constraint:
\begin{equation} \label{prob:Isoperimetric}
\begin{split}
  &\min_{u \in \overline{u} + W^{1, 2}_0(E; \mathbb{R}^d)} \enspace 
  \mathcal{I}_0(u) = \int_E f_0(u(x), \nabla u(x), x) \, dx
  \\
  &\text{subject to} \enspace \mathcal{I}_1(u) = \int_E f_1(u(x), \nabla u(x), x) \, dx - \zeta = 0.
\end{split}
\end{equation}
Here $f_i \colon \mathbb{R}^d \times \mathbb{R}^{d \times n} \times E \to \mathbb{R}$, $f_i = f_i(u, \xi, x)$ are
Carath\'{e}odory functions that are differentiable in $u$ and $\xi$ and their partial derivatives $\nabla_u f_i$ and
$\nabla_{\xi} f_i$ are Carath\'{e}odory functions as well, while $\zeta \in \mathbb{R}$ is a given constant. Although
without loss of generality one can suppose that $\zeta = 0$, in many applications it is more convenient to consider the
case $\zeta \ne 0$, instead of redefining the function $f_1$. Finally, the function 
$\overline{u} \in W^{1, 2}(E; \mathbb{R}^d)$ defines boundary conditions, since
\[
  \{ \overline{u}  \} + W^{1, 2}_0(E; \mathbb{R}^d) 
  = \Big\{ u \in W^{1, 2}(E; \mathbb{R}^d) \Bigm| \trace(u) = \trace(\overline{u}) \Big\},
\]
where $\trace(\cdot)$ is the trace operator (see \cite[Sects.~5.20--5.22]{Adams}).

We impose the following growth conditions on the functions $f_i$ and their partial derivatives. Namely, there exist 
$C > 0$ and a.e. nonnegative functions $\eta \in L^1(E)$ and $\theta \in L^2(E)$ such that
\begin{equation} \label{eq:MultiDimGrowthCond}
\begin{split}
  |f_i(u, \xi, x)| &\le C\big( |u|^2 + |\xi|^2 \big) + \eta(x), 
  \\
  \max\big\{ |\nabla_u f_i(u, \xi, x)|, |\nabla_{\xi} f_i(u, \xi, x)| \big\} 
  &\le C\big( |u| + |\xi| \big) + \theta(x).
\end{split}
\end{equation}
for a.e. $x \in E$ and any $u \in \mathbb{R}^d$, $\xi \in \mathbb{R}^{d \times n}$ and $i \in \{ 0,  1 \}$. Let us
note that these growth conditions can be relaxed with the use of the Sobolev imbedding theorem \cite[Thm.~5.4]{Adams}.
In particular, in the case $n = 1$ it is sufficient to assume that the growth conditions of the form
\eqref{eq:1dimGrowthCond} hold true. The first inequality in \eqref{eq:MultiDimGrowthCond} can be replaced with
$|f_i(u, \xi, x)| \le C\big( |u|^q + |\xi|^2 \big) + \eta(x)$, where $q = 2n/(n - 2)$ in the case $n > 2$, and 
$q$ is any number in $[2, + \infty)$ in the case $n = 2$, etc.

The growth conditions \eqref{eq:MultiDimGrowthCond} ensure that the functionals $\mathcal{I}_i$, $i \in \{ 0, 1 \}$, are
G\^{a}teaux differentiable and their G\^{a}teaux derivatives have the form:
\begin{equation} \label{eq:MultiDimVarFuncDeriv}
  \mathcal{I}_i'(u)[w] = \int_E \big( \langle \nabla_u f_i(u(x), u'(x), x), w(x) \rangle 
  + \langle \nabla_{\xi} f_i(u(x), u'(x), x), \nabla w(x) \rangle \big) \, dx
\end{equation}
for all $w \in W^{1, 2}(E; \mathbb{R}^d)$ (see, e.g. \cite[Sect.~3.4.2]{Dacorogna}). Arguing in the same way as in
the previous subsection and utilising the standard results on the continuity of Nemytskii operators
\cite{AppellZabrejko} one can verify that the growth conditions \eqref{eq:MultiDimGrowthCond} also guarantee that 
the G\^{a}teaux derivatives $\mathcal{I}'_i(\cdot)$ are continuous and, therefore, the functionals $\mathcal{I}_i$, 
$i \in \{ 0, 1 \}$, are continuously Fr\'{e}chet differentiable.

In order to compute the gradients of the functionals $\mathcal{I}_i$, we will utilise an extension of the ``transition
into the space of derivatives'' technique to the multidimensional case, developed by the author in
\cite{Dolgopolik_MultidimCalcVar}. To this end, suppose that for all KKT points $(u_*, \lambda_*)$ of problem
\eqref{prob:Isoperimetric} one has $u_* \in \overline{u} + MW_0^{n, 2}(E; \mathbb{R}^d)$ (see
Subsection~\ref{subsect:SobolevSpaces}). In other words, suppose that solutions of the corresponding Euler-Lagrange
equation are more regular than is assumed in the formulation of the problem, in the sense that there exist all weak
mixed derivatives of these solutions of all orders $k \in \{ 2, \ldots, n \}$ and these derivatives belong to
$L^2(E)$. Then problem \eqref{prob:Isoperimetric} is equivalent to the following variational problem:
\begin{equation} \label{prob:IsoperimetricRegular}
\begin{split}
  &\min_{u \in MW^{n, 2}_0(E; \mathbb{R}^d)} \enspace 
  \mathcal{I}_0(\overline{u} + u) 
  = \int_E f_0(\overline{u}(x) + u(x), \nabla \overline{u}(x) + \nabla u(x), x) \, dx
  \\
  &\text{s.t.} \enspace 
  \mathcal{I}_1(\overline{u} + u) = 
  \int_E f_1(\overline{u}(x) + u(x), \nabla \overline{u}(x) + \nabla u(x), x) \, dx - \zeta = 0.
\end{split}
\end{equation}
By Theorem~\ref{thm:MixedSobolev} this problem is equivalent to the following one:
\begin{equation} \label{prob:IsoperimetricDerivativeSpace}
  \min_{v \in X} \enspace 
  \widehat{\mathcal{I}}_0(v) := \mathcal{I}_0(\overline{u} + \mathcal{A} v)
  \quad
  \text{subject to} \enspace 
  \widehat{\mathcal{I}}_1(v) := \mathcal{I}_1(\overline{u} + \mathcal{A} v) = \zeta,
\end{equation}
where the linear operator $\mathcal{A}$ is defined in \eqref{eq:MultidimIntegralOper} and
\begin{equation} \label{eq:DerivSpace_MultiDim}
  X = \left\{ v \in L^2(E; \mathbb{R}^d) \biggm| \int_{a_i}^{b_i} v \, ds_i = 0 \text{ for a.e. } x \in E 
  \enspace \forall i \in \{ 1, \ldots, n \} \right\}.
\end{equation}
Applying \eqref{eq:MultiDimVarFuncDeriv} and integrating by parts (see \cite{Dolgopolik_MultidimCalcVar} for more
details) one obtains that the functionals $\widehat{\mathcal{I}}_i$, $i \in \{ 0, 1 \}$, are continuously Fr\'{e}chet
differentiable and their Fr\'{e}chet derivatives have the form
\[
  D \widehat{\mathcal{I}}_i[v](h) = \int_E \langle P_i(v)(x), h(x) \rangle dx \quad \forall h \in X,
\]
where
\begin{align} \notag
  P_i(v)(x) = &(-1)^n \int_{x_1}^{b_1} \ldots \int_{x_n}^{b_n} \frac{\partial f}{\partial u}(u(s), \nabla u(s), s) \, ds
  \\ \label{eq:MultidimCV_Grad}
  &\begin{aligned}[t]+ (-1)^{n - 1} \sum_{i = 1}^n 
    &\int_{x_n}^{b_n} \ldots \int_{x_{i + 1}}^{b_{i + 1}} \int_{x_{i - 1}}^{b_{i - 1}} \ldots \int_{x_1}^{b_1}
    \\ 
    &\frac{\partial f}{\partial \xi_i} (u(s^i), \nabla u(s^i), s^i) \, ds_1 \ldots ds_{i - 1} ds_{i + 1} \ldots ds_n
  \end{aligned}
\end{align}
for a.e. $x \in E$, $u = \overline{u} + \mathcal{A} v$, 
$\partial f / \partial \xi_i = (\partial f / \partial \xi_{1i}, \ldots, \partial f / \partial \xi_{di})$, and
$s^i = (s_1, \ldots, s_{i - 1}, x_i, s_{i + 1}, \ldots, s_n)$ for any $i \in \{ 1, \ldots, n \}$. With the use of the
growth conditions \eqref{eq:MultiDimGrowthCond} one can readily check that the nonlinear operator $P_i(\cdot)$
continuously maps $L^2(E)$ to $L^2(E)$. Therefore, the gradient of the functional $\widehat{\mathcal{I}}_i$ at
a point $v \in X$ in the Hilbert space $X$ is the orthogonal projection $\proj_X (P_i(v))$ of the function $P_i(v)$ in
$L^2(E)$ onto the subspace $X$. This projection can be computed analytically with the use of the following
proposition (see \cite[Prop.~3]{Dolgopolik_MultidimCalcVar}).

\begin{proposition}
For any $v \in L^2(E)$ one has
\[
  \proj_X v = v + \sum_{k = 1}^n \sum_{\alpha \in I_k} (- 1)^{|\alpha|} c_{\alpha} S_{\alpha} v,
\]
where for any $\alpha \in I_k$ one has $c_{\alpha} = \prod_{i = 1}^n (b_i - a_i)^{-\alpha_i}$ and
\[
  \Big( S_{\alpha} v \Big)(x) 
  = \int_{a_{\alpha_{i_1}}}^{b_{\alpha_{i_1}}} \ldots \int_{a_{\alpha_{i_k}}}^{b_{\alpha_{i_k}}}
  v \, ds_{i_k} \ldots ds_{i_1},
\]
and the indices $1 \le i_1 < \ldots < i_k \le n$ are such that $\alpha_r = 0$ if and only if $r \ne i_k$ (that is, $i_j$
are indices of nonzero components of the multi-index $\alpha$).
\end{proposition}

Thus, $\nabla \widehat{\mathcal{I}}_i(v) = \proj_X P_i(v)$, $i \in \{ 0, 1 \}$. The exact augmented Lagrangian
$\mathscr{L}(v, \lambda, c)$ with $\lambda \in \mathbb{R}$ for problem \eqref{prob:IsoperimetricDerivativeSpace} is
defined according to equalities \eqref{eq:ExactAugmLagrDef} and \eqref{eq:KKTPenTerm} as follows:
\begin{align*}
  \mathscr{L}(v, \lambda, c) = &\int_E f_0(u(x), \nabla u(x), x) \, dx 
  + \lambda \bigg( \int_E f_1(u(x), \nabla u(x), x) \, dx - \zeta \bigg)
  \\
  &+ \frac{c}{2} (1 + \lambda^2) \bigg( \int_E f_1(u(x), \nabla u(x), x) \, dx - \zeta \bigg)^2
  \\
  &+ \frac{1}{2} \bigg( \int_E \langle \proj_X P_1(v)(x), \proj_X P_0(v)(x) + \lambda \proj_X P_1(v)(x) \rangle 
  \, dx \bigg)^2  
\end{align*}
where $u = \overline{u} + \mathcal{A} v$ or, equivalently, $v = D^{(1, \ldots, 1)}(u - \overline{u})$ in the weak sense.

Let us provide sufficient conditions for the exactness of the augmented Lagrangian $\mathscr{L}(v, \lambda, c)$. We will
formulate these conditions in ``abstract terms'' and then show how these conditions can be reformulated in terms of
assumptions on the functions $f_0$ and $f_1$ in some particular cases.

\begin{theorem} \label{thm:Isoperimetric_ExAugmLagr}
Let the following assumptions be valid:
\begin{enumerate}
\item{for all KKT points $(u_*, \lambda_*)$ of problem \eqref{prob:Isoperimetric} one has 
$u_* \in \overline{u} + MW_0^{n, 2}(E; \mathbb{R}^d)$;
\label{hyp:Iso1}}

\item{the growth conditions \eqref{eq:MultiDimGrowthCond} hold true; 
\label{hyp:Iso2}}

\item{the functions $f_0$ and $f_1$ are twice differentiable in $u$ and $\xi$ and all their corresponding second order
derivatives are Carath\'{e}odory functions that are essentially bounded on 
$\mathbb{R}^d \times \mathbb{R}^{d \times n} \times \mathbb{R}^n$;
\label{hyp:Iso3}}

\item{the augmented Lagrangian $\mathscr{L}(v, \lambda, c)$ is weakly sequentially l.s.c. in $(v, \lambda)$;}

\item{for any $\gamma \in \mathbb{R}$ there exists $c > 0$ such that the set 
\[
  \Omega_c(\gamma) = \Big\{ u \in \overline{u} + MW_0^{n, 2}(E; \mathbb{R}^d) \Bigm|
  \mathcal{I}_0(u) + c \big( \mathcal{I}_1(u) - \zeta \big)^2 \le \gamma \Big\}
\]
is bounded in $MW^{n, 2}(E; \mathbb{R}^d)$;}

\item{for any bounded subset $V \subset X$ there exists $\sigma > 0$ such that the inequality 
$\| \proj_X P_1(v) \|_2 \ge \sigma$ holds true for all $v \in V$.}
\end{enumerate}
Then the augmented Lagrangian $\mathscr{L}(v, \lambda, c)$ is globally exact and for any $\gamma \in \mathbb{R}$ there
exists $c(\gamma) > 0$ such that for all $c \ge c(\gamma)$ the following statements hold true:
\begin{enumerate}
\item{the sublevel set $S_c(\gamma)$ is bounded in $X$;}

\item{if $(v_*, \lambda_*) \in S_c(\gamma)$ is a point of local minimum of $\mathscr{L}(v, \lambda, c)$, then 
$u_* = \overline{u} + \mathcal{A} v_*$ is locally optimal solution of problem \eqref{prob:Isoperimetric};}

\item{a pair $(v_*, \lambda_*) \in S_c(\gamma)$ is a stationary point of $\mathscr{L}(v, \lambda, c)$ if and only if
$(u_*, \lambda_*)$ with $u_* = \overline{u} + \mathcal{A} v_*$ is a KKT point of problem
\eqref{prob:Isoperimetric} such that $\mathcal{I}_0(u_*) \le \gamma$.
}
\end{enumerate}
\end{theorem}

\begin{proof}
The first assumption of the theorem guarantees that the isoperimetric problem \eqref{prob:Isoperimetric} in the Sobolev
space $W^{1, 2}(E; \mathbb{R}^d)$ is equivalent to this problem formulated in the space 
$MW^{n, 2}(E; \mathbb{R}^d)$ (see \eqref{prob:IsoperimetricRegular}), which is, in turn, equivalent to problem
\eqref{prob:IsoperimetricDerivativeSpace} by Theorem~\ref{thm:MixedSobolev}.

The second and third assumptions of the theorem ensure that the functionals $\widehat{\mathcal{I}}_0$ and
$\widehat{\mathcal{I}}_1$ are twice continuously Fr\'{e}chet differentiable and their first and second order Fr\'{e}chet
derivatives are bounded on bounded subsets of the space $X$.

Note that in the case of problem \eqref{prob:IsoperimetricDerivativeSpace} one has
$Q(v)[\lambda] = (\| \proj_X P_1(v) \|_2)^4 \lambda^2$ (see \eqref{eq:Q_def}). Consequently, the last
assumption of the theorem implies that the function $Q(\cdot)$ is uniformly positive definite on any bounded subset of
the space $X$. Therefore, applying the fifth assumption of the theorem and 
\cite[Cor.~3.6 and Thms.~5.3 and 5.7]{Dolgopolik_ExAugmLagr} one obtains that all three statements of the theorem hold
true.
\end{proof}

\begin{remark}
Let us comment on the assumptions of the previous theorem:

\noindent{(i)~With the use of the Sobolev imbedding theorem one can relax the assumption on the essential boundedness of
second order derivatives of the functions $f_0$ and $f_1$ in some particular cases. For example, in the case $n = 1$
it is sufficient to suppose that for any $R > 0$ there exist $C_R > 0$ and a.e. nonnegative functions 
$\eta_R \in L^1(E)$ and $\theta_R \in L^2(E)$ such that
\begin{gather*}
  \left| \frac{\partial^2 f}{\partial u^2}(u, \xi, x) \right| \le C_R |\xi|^2 + \eta_R(x), \quad
  \left| \frac{\partial^2 f}{\partial u \partial \xi}(u, \xi, x) \right| \le C_R |\xi| + \theta_R(x),
  \\
  \left| \frac{\partial^2 f}{\partial \xi^2}(u, \xi, x) \right| \le C_R
\end{gather*}
for a.e. $x \in E$ and  all $(u, \xi) \in \mathbb{R}^{2d}$ such that $|u| \le R$.
}

\noindent{(ii)~The augmented Lagrangian $\mathscr{L}(v, \lambda, c)$ is weakly sequentially l.s.c. in $(v, \lambda)$, if
$f_0$ is convex in $\xi$ and quadratic in $(u, \xi)$, $f_1$ is affine in $(u, \xi)$, and the functions $f_0$ and $f_1$
satisfy some standard growth conditions. Under these assumptions, by \cite[Cor.~3.24 and Thm.~3.26]{Dacorogna} the
functional $\mathcal{I}_0(\cdot)$ is weakly sequentially l.s.c., the functional $\mathcal{I}_1(\cdot)$ is weakly
sequentially continuous, while the function 
\[
  (v, \lambda) \mapsto \langle \proj_X P_1(v)(\cdot), \proj_X P_0(v)(\cdot) + \lambda \proj_X P_1(v)(\cdot) \rangle
\]
is affine in $v$ and also weakly sequentially continuous (see \cite{Dacorogna_WeakCont}). With the use of these facts
one can readily verify that $\mathscr{L}(v, \lambda, c)$ is weakly sequentially l.s.c. in $(v, \lambda)$. It should be
mentioned that in various particular cases (e.g. the case $n = 1$) one can prove the weak sequential l.s.c. of 
the augmented Lagrangian under less restrictive assumptions.
}

\noindent{(iii)~In the case $n = 1$, the fifth assumption of the theorem is satisfied, provided
\begin{equation} \label{eq:CalcVar_CoercivityCond}
  f_0(u, \xi, x) \ge C_1 |\xi|^2 + C_2 |u|^q + \eta(x) \quad \forall (u, \xi, x) \in \mathbb{R}^{2d} \times E
\end{equation}
for some $C_1 > 0$, $C_2 \in \mathbb{R}$, $1 \le q < 2$, and $\eta \in L^1(E)$, since this inequality implies that
the functional $u \mapsto \mathcal{I}_0(\overline{u} + u)$ is coercive on 
$W^{1, 2}_0(E; \mathbb{R}^d) = MW_0^{1, 2}(E; \mathbb{R}^d)$. In the case $n > 1$, the question of when the
sublevel set $\Omega_c(\gamma)$ of the penalty function $\mathcal{I}_0(\cdot) + c (\mathcal{I}_1(\cdot) - \zeta)^2$ is
actually bounded in $MW_0^{n, 2}(E; \mathbb{R}^d)$ is a challenging open problem for future research.
}

\noindent{(iv)~Let us point out some particular cases in which the last assumption of the theorem holds true. If the
function $f_1$ is affine in $(u, \xi)$, that is, 
$f_1(u, \xi, x) = \langle h_1(x), \xi \rangle + \langle h_2(x), u \rangle + h_3(x)$, then the operator $P_1(\cdot)$ is
constant and the fourth assumption of the theorem holds true, provided the projection of $P_1(\cdot)$ onto $X$ is
nonzero. Note that by \cite[Prop.~5]{Dolgopolik_MultidimCalcVar} the assumption on the projection of $P_1$ is satisfied
in the case when $h_1$ is differentiable, provided the function $h_2(\cdot) - \sum_{i = 1}^n D_i h_1(\cdot)$ is not
identically zero. In the case when the function $f_1$ is quadratic in $(u, \xi)$, the fourth assumption of the theorem
is satisfied, if $\proj_X P_1(v) \ne 0$ for any $v \in X$, since in this case the function $P_1(\cdot)$ is affine, which
implies that the norm $\| \proj_X P_1(\cdot) \|_2$ is weakly lower semicontinuous. With the use of this fact one can
prove the existence of the required $\sigma > 0$. Note that by \cite[Prop.~5]{Dolgopolik_MultidimCalcVar} under some
smoothness assumptions the inequality $\proj_X P_1(v) \ne 0$ means that the Euler-Lagrange equation for the functional
$\mathcal{I}_1$ has no solutions. In the general case, the last assumption of the theorem means that the gradient of
the functional $\mathcal{I}_1$ is bounded away from zero on bounded sets, which under some smoothness assumptions 
implies that the Euler-Lagrange equation for $\mathcal{I}_1(\cdot)$ has no solutions.
}
\end{remark}

Let us provide sufficient conditions for the complete exactness of the augmented Lagrangian $\mathscr{L}(v, \lambda, c)$
that do note involve the restrictive assumption on the boundedness of the sublevel set $\Omega_c(\gamma)$ in the space 
$MW_0^{n, 2}(E; \mathbb{R}^d)$ and the weak sequential lower semicontinuity of the augmented Lagrangian.

\begin{theorem} \label{thm:Isoperimetric_ExAugmLagr_Stronger}
Let assumptions~\ref{hyp:Iso1}--\ref{hyp:Iso3} of Theorem~\ref{thm:Isoperimetric_ExAugmLagr} and the following
assumptions hold true:
\begin{enumerate}
\item{for any $\gamma \in \mathbb{R}$ there exists $r = r(\gamma) > 0$ such that the set
\begin{equation} \label{eq:IsoPenFuncSublevelSet}
  Z_r(\gamma) := \Big\{ u \in W^{1, 2}_0(E; \mathbb{R}^d) \Bigm| 
  \mathcal{I}_0(\overline{u} + u) + r (\mathcal{I}_1(\overline{u} + u) - \zeta)^2 \le gamma \}
\end{equation}
is bounded in $W^{1, 2}(E; \mathbb{R}^d)$;
\label{hyp:Iso2nd}}

\item{for any subset $V \subset X$ such that the set $\{ u = \mathcal{A} v \mid v \in V \}$ is bounded in
$W^{1, 2}_0(E; \mathbb{R}^d)$ there exists $\sigma > 0$ such that the inequality 
$\| \proj_X P_1(v) \|_2 \ge \sigma$ holds true for all $v \in V$;
}
\end{enumerate}
Then for the augmented Lagrangian $\mathscr{L}(v, \lambda, c)$ to be completely exact on the set $S_c(\gamma)$ for any
$\gamma \in \mathbb{R}$ it is necessary and sufficient that there exists an unbounded increasing sequence 
$\{ c_k \} \subset (0, + \infty)$ such that for any $k \in \mathbb{N}$ the augmented Lagrangian 
$\mathscr{L}(\cdot, c_k)$ attains a global minimum on the set $X \times \mathbb{R}$.
\end{theorem}

\begin{proof}
The second and third assumptions of Theorem~\ref{thm:Isoperimetric_ExAugmLagr} guarantee that the functionals
$\widehat{\mathcal{I}}_0$ and $\widehat{\mathcal{I}}_1$ are twice continuously Fr\'{e}chet differentiable. Furthremore, 
the growth conditions \eqref{eq:MultiDimGrowthCond} along with the essential boundedness of the second order
derivatives of $f_0$ and $f_1$ in $(u, \xi)$ imply that the the functionals $\widehat{\mathcal{I}}_0$ and
$\widehat{\mathcal{I}}_1$ and their first and second order derivatives are bounded on any set $V \subset X$ such that
the set $\{ u = \mathcal{A} v \mid v \in V \}$ is bounded in the Sobolev space $W^{1, 2}(E; \mathbb{R}^d)$. In
particular, for any $\gamma > 0$ they are bounded on the set
\[
  V = \Big\{ v \in X \Big\{ v \in X \Bigm| 
  \widehat{\mathcal{I}}_0(v) + r(\gamma) \big( \widehat{\mathcal{I}}_1(v) - \zeta)^2 \le gamma \Big\},
\]
since $\overline{u} + \{ u = \mathcal{A} v \mid v \in V \} \subset Z_{r(\gamma)}(\gamma)$ and by our assumption the set
$Z_{r(\gamma)}(\gamma)$ is bounded in $W^{1, 2}(E; \mathbb{R}^d)$.

In turn, the assumption on $\| \proj_X P_1(v) \|_2$ ensures that for any $\gamma > 0$ there exists $\sigma(\gamma) > 0$
such that $\sigma_{\max}(Q(v)) \ge \sigma(\gamma)$ for all $v \in V$. Hence if $\mathscr{L}(\cdot, c_k)$ attains a
global minimum on the set $X \times \mathbb{R}$ for some increasing unbounded sequence $\{ c_k \}$, then by
Theorem~\ref{THM:COMPLETE_EXACTNESS} the augmented Lagrangian $\mathscr{L}(v, \lambda, c)$ is completely exact on the
set $S_c(\gamma)$. 

Conversely, if the augmented Lagrangian is completely exact, then by the definition of complete exactness it must attain
a global minimum on $X \times \mathbb{R}$ for any sufficiently large $c > 0$.
\end{proof}

\begin{remark}
The sublevel set \eqref{eq:IsoPenFuncSublevelSet} is bounded, in particular, if there exist $C_1 > 0$, $C_2 > 0$, 
$q \in [1, p)$, and $\eta \in L^1(E)$ such that
\[
  f_0(u, \xi, x) \ge C_1 |\xi|^2 + C_2 |u|^q + \eta(x)
\]
for a.e. $x \in E$ and all $(u, \xi) \in \mathbb{R}^d \times \mathbb{R}^{d \times n}$, since this growth condition
ensures that the functional $\mathcal{I}_0(\cdot)$ is coercive (see, e.g. \cite[Thm.~3.30]{Dacorogna}).
\end{remark}

%

\subsection{Problems with nonholonomic equality constraints}
\label{subsect:Nonholonomic}

In this subsection, we study an exact augmented Lagrangian for multidimensional variational problems with nonholonomic
equality constraints (which can be viewed as first order PDE constraints). Let, as in the previous section, 
$E = \prod_{i = 1}^n (a_i, b_i) \subset \mathbb{R}^n$ be a bounded open box and 
$\overline{u} \in W^{1, 2}(E; \mathbb{R}^d)$ be a given function. Consider the following variational problem with
nonholonomic equality (first order PDE) constraints:
\begin{equation} \label{prob:Nonholonomic}
\begin{split}
  &\min_{u \in \overline{u} + W^{1, 2}_0(E; \mathbb{R}^d)} \enspace 
  \mathcal{I}(u) = \int_E f(u(x), \nabla u(x), x) \, dx
  \\
  &\text{subject to} \enspace F(u(x), \nabla u(x), x) = 0 \text{ for a.e. } x \in E.
\end{split}
\end{equation}
Here $f \colon \mathbb{R}^d \times \mathbb{R}^{d \times n} \times E \to \mathbb{R}$, $f = f(u, \xi, x)$,
and $F \colon \mathbb{R}^d \times \mathbb{R}^{d \times n} \times E \to \mathbb{R}^{\ell}$, $F = F(u, \xi, x)$, are
Carath\'{e}odory functions that are differentiable in $u$ and $\xi$ and their partial derivatives in $u$ and $\xi$ are
Carath\'{e}odory functions as well.

Suppose that the function $f$ satisfies the growth conditions of the form \eqref{eq:MultiDimGrowthCond}, while the
function $F$ satisfies the following growth conditions. Namely, there exist $C > 0$ and an a.e. nonnegative function
$\eta \in L^2(E)$ such that
\begin{equation} \label{eq:NemytskiiGrowthCond}
  |F(u, \xi, x)| \le C \big( |u| + |\xi| \big) + \eta(x), \quad
  \max\big\{ |\nabla_u F(u, \xi, x)|, |\nabla_{\xi} F(u, \xi, x)| \big\} \le C.
\end{equation}
These assumptions guarantee that the functional $\mathcal{I}(\cdot)$ is continuously Fr\'{e}chet differentiable,
while the Nemytskii operator $\mathcal{F}(u) = F(u(\cdot), \nabla u(\cdot), \cdot)$ maps the Sobolev space
$W^{1, 2}(E; \mathbb{R}^d)$ to $L^2(E; \mathbb{R}^{\ell})$ and is continuously Fr\'{e}chet differentiable as
well (see, e.g. \cite{AppellZabrejko}). The Fr\'{e}chet derivative of $\mathcal{I}$ has the form
\eqref{eq:MultiDimVarFuncDeriv}, while the Fr\'{e}chet derivative of the Nemytskii operator $\mathcal{F}$ is defined as
\begin{equation} \label{eq:NemytskiiOperatorDeriv}
  D \mathcal{F}(u)[w] = \nabla_u F(u(\cdot), \nabla u(\cdot), \cdot) w(\cdot)
  + \sum_{i = 1}^n \frac{\partial F}{\partial \xi^i}(u(\cdot), \nabla u(\cdot), \cdot) \nabla w_i(\cdot)
\end{equation}
for any $u, w \in W^{1, 2}(E; \mathbb{R}^d)$, where 
$\partial/\partial \xi^i = (\partial/\partial \xi_{i1}, \ldots, \partial / \partial \xi_{in})$ and
$w = (w_1, \ldots, w_d)$.

In order to define an exact augmented Lagrangian for problem \eqref{prob:Nonholonomic} we will use the same technique as
in the case of isoperimetric constraints, that is, we impose an additional assumption on the regularity of minimizers
of problem \eqref{prob:Nonholonomic}. Namely, suppose that all KKT points $(u_*, \lambda_*)$ of problem 
\eqref{prob:Nonholonomic} satisfy the condition $u_* \in \overline{u} + MW_0^{n, 2}(E; \mathbb{R}^d)$. Then taking into
account Theorem~\ref{thm:MixedSobolev} one can conclude that problem \eqref{prob:Nonholonomic} is equivalent to the
following one:
\begin{equation} \label{prob:NonholonomicDerivativeSpace}
\begin{split}
  \min_{v \in X} \enspace 
  \widehat{\mathcal{I}}(v) := \mathcal{I}(\overline{u} + \mathcal{A} v) 
  \enspace
  \text{subject to} \enspace \widehat{\mathcal{F}}(u)
  := \mathcal{F}(\overline{u} + \mathcal{A} v)  = 0,
\end{split}
\end{equation}
where the Hilbert space $X$ is defined in \eqref{eq:DerivSpace_MultiDim}. Arguing in the same way as in previous
section, one obtains that the gradient of the functional $\widehat{\mathcal{I}}$ at a point $v \in X$ has the form
$\proj_X P(v)$, where $P(v)$ is defined according to the equality \eqref{eq:MultidimCV_Grad} with $f_i$ replaced by
$f$. In turn, the Nemytskii operator $\widehat{\mathcal{F}}(\cdot)$ maps $X$ to $L^2(E; \mathbb{R}^{\ell})$ and is
continuously Fr\'{e}chet differentiable. Applying \eqref{eq:NemytskiiOperatorDeriv} and integrating by parts one obtains
that the Fr\'{e}chet derivative of $\widehat{\mathcal{F}}(\cdot)$ has the form 
$D \widehat{\mathcal{F}}(v)[h] = P_F(v) h$ for any $v, h \in X$, where 
\begin{align*}
  P_F(v)(x) = &(-1)^n \int_{x_1}^{b_1} \ldots \int_{x_n}^{b_n} \nabla_u F(u(s), \nabla u(s), s) \, ds
  \\
  &\begin{aligned}[t]
    + (-1)^{n - 1} \sum_{i = 1}^n 
    &\int_{x_n}^{b_n} \ldots \int_{x_{i + 1}}^{b_{i + 1}} \int_{x_{i - 1}}^{b_{i - 1}} \ldots \int_{x_1}^{b_1}
    \\ 
    &\frac{\partial F}{\partial \xi_i} (u(s^i), \nabla u(s^i), s^i) \, ds_1 \ldots ds_{i - 1} ds_{i + 1} \ldots ds_n
  \end{aligned}
\end{align*}
for a.e. $x \in E$, where $u = \overline{u} + \mathcal{A} v$. Note that the growth conditions
\eqref{eq:NemytskiiGrowthCond} imply that $P_F(v) \in L^{\infty}(E; \mathbb{R}^{\ell \times d})$ for any $v \in X$.

The exact augmented Lagrangian $\mathscr{L}(v, \lambda, c)$ with $\lambda \in L^2(E; \mathbb{R}^{\ell})$ for
problem \eqref{prob:NonholonomicDerivativeSpace} is defined according to equalities \eqref{eq:ExactAugmLagrDef} and
\eqref{eq:KKTPenTerm} as follows:
\begin{align*}
  \mathscr{L}(v, \lambda, c) = &\int_E f(u(x), \nabla u(x), x) \, dx 
  + \int_E \langle \lambda(x), F(u(x), \nabla u(x), x) \rangle \, dx
  \\
  &+ \frac{c}{2} (1 + \| \lambda \|_2^2) \int_E |F(u(x), \nabla u(x), x)|^2 \, dx
  \\
  &+ \frac{1}{2} \int_E \big| P_F(v)(x)\big( \proj_X P(v)(x) + P_F(v)(x)^* \lambda(x) \big) \Big|^2 \, dx  
\end{align*}
where $u = \overline{u} + \mathcal{A} v$ and $P_F(v)(x)^*$ is the transpose of the matrix $P_F(v)(x)$.

\begin{remark}
Note that in this section we assume that Lagrange multipliers $\lambda$ belong to $L^2(E; \mathbb{R}^{\ell})$. However,
in the case multidimensional variational problems even multipliers $\lambda \in L^1(E; \mathbb{R}^{\ell})$ do not always
exist (see \cite[Sect.~2.3]{GiaquintaHildebrandt}). Therefore, in order to make the general theory of exact augmented
Lagrangian applicable to variational problems with nonholonomic constraints, in this section we consider only those
problems with nonholonomic constraints for which multipliers $\lambda \in L^2(E; \mathbb{R}^{\ell})$ exist for all
globally optimal solutions (see Remark~\ref{rmrk:BasicExistenceAssumption}). In addition, it should be underlined that
in the formulations of the theorems below a pair $(u_*, \lambda_*)$ is a KKT point of problem \eqref{prob:Nonholonomic},
only if $\lambda_* \in L^2(E; \mathbb{R}^{\ell})$. Thus, if there is a point $u_*$ satisfying KKT optimality conditions
for problem \eqref{prob:Nonholonomic} with a multiplier $\lambda_*$ that does not belong to $L^2(E; \mathbb{R}^{\ell})$,
then within our theory this point $u_*$ is \textit{not} stationary and the pair $(u_*, \lambda_*)$ is \textit{not} a KKT
point of problem \eqref{prob:Nonholonomic}.
\end{remark}

Let us provide sufficient conditions for the exactness of the augmented Lagrangian $\mathscr{L}(v, \lambda, c)$. For the
sake of simplicity, we will formulate these conditions under the assumption that the mapping $F$ is affine in $(u, \xi)$
(that is, problem \eqref{prob:Nonholonomic} is a variational problem with constraints defined by a system of linear
PDE), since this is the simplest assumption under which the corresponding Nemytskii operator $\mathcal{F}(\cdot)$ is
twice continuously Fr\'{e}chet differentiable.

\begin{theorem} \label{thrm:Nonholonomic}
Let the following assumptions be valid:
\begin{enumerate}
\item{for all KKT points $(u_*, \lambda_*)$ of problem \eqref{prob:Nonholonomic} one has 
$u_* \in \overline{u} + MW_0^{n, 2}(E; \mathbb{R}^d)$;}

\item{the function $f$ is twice differentiable in $(u, \xi)$, satisfies the growth conditions
\eqref{eq:MultiDimGrowthCond}, and its second order derivatives are Carath\'{e}odory functions that are essentially
bounded on $\mathbb{R}^d \times \mathbb{R}^{d \times n} \times \mathbb{R}^n$;}

\item{the augmented Lagrangian $\mathscr{L}(v, \lambda, c)$ is weakly sequentially l.s.c. in $(v, \lambda)$;}

\item{the function $F$ has the form $F(u, \xi, x) = A(x) u + \sum_{i = 1}^n B_i(x) \xi_i + D(x)$ for some essentially
bounded measurable functions $A$, $B_i$ and $D$, and $P_F(\cdot) \ne 0$ (if the functions $B_i$ are continuously
differentiable, it is sufficient to suppose that $A(x) - \sum_{i = 1}^n D_i B_i(x) \ne 0$ on a set of positive
measure);}

\item{for any $\gamma \in \mathbb{R}$ there exists $c > 0$ such that the set 
\[
  \Omega_c(\gamma) = \Big\{ u \in \overline{u} + MW_0^{n, 2}(E; \mathbb{R}^d) \Bigm|
  \mathcal{I}_0(u) + c \| F(u(\cdot), \nabla u(\cdot), \cdot) \|_2^2 \le \gamma \Big\}
\]
is bounded in $MW^{n, 2}(E; \mathbb{R}^d)$.}
\end{enumerate}
Then the augmented Lagrangian $\mathscr{L}(v, \lambda, c)$ is globally exact and for any $\gamma \in \mathbb{R}$ there
exists $c(\gamma) > 0$ such that for all $c \ge c(\gamma)$ the following statements hold true:
\begin{enumerate}
\item{the sublevel set $S_c(\gamma)$ is bounded in $X$;}

\item{if $(v_*, \lambda_*) \in S_c(\gamma)$ is a point of local minimum of $\mathscr{L}(v, \lambda, c)$, then 
$u_* = \overline{u} + \mathcal{A} v_*$ is locally optimal solution of problem \eqref{prob:Nonholonomic};}

\item{a pair $(v_*, \lambda_*) \in S_c(\gamma)$ is a stationary point of $\mathscr{L}(v, \lambda, c)$ if and only if
$(u_*, \lambda_*)$ with $u_* = \overline{u} + \mathcal{A} v_*$ is a KKT point of problem
\eqref{prob:Nonholonomic} such that $\mathcal{I}(u_*) \le \gamma$.
}
\end{enumerate}
\end{theorem}

The proof of this theorem largely repeats the proof of Theorem~\ref{thm:Isoperimetric_ExAugmLagr} and consists in
verifying that the assumptions of the theorem ensure that the assumptions of 
\cite[Cor.~3.6 and Thms.~5.3 and 5.7]{Dolgopolik_ExAugmLagr} are satisfied. We omit it for the sake of shortness.

\begin{remark}
Let us note that the augmented Lagrangian $\mathscr{L}(v, \lambda, c)$ is weakly sequentially l.s.c. in $(v, \lambda)$,
in particular, if $n = 1$, the function $F = F(u, \xi, x)$ does not depends on $\xi$ (i.e. the constraints are
holonomic), and the function $f$ has the form
\[
  f(u, \xi, x) = \langle \xi, h_2(x) \xi \rangle + \langle h_1(x), \xi) \rangle + h_0(u, x)
\]
for some functions $h_2$, $h_1$, and $h_0$ such that the matrix $h_2(x)$ is positive semidefinite for a.e. 
$x \in E$. Under these assumptions the functional $\mathcal{I}(\cdot)$ is weakly sequentially lower semicontinuous
(provided $h_2$ satisfies appropriate growth conditions) by \cite[Cor.~3.24]{Dacorogna}, the map $\proj_X P(\cdot)$ is
weakly sequentially continuous, the map $P_F(\cdot)$ does not depends on $v$, while the Nemytskii operator 
$u \mapsto F(u(\cdot), \nabla u(\cdot), \cdot)$ maps weakly continuous sequences to strongly continuous ones by the
Rellich-Kondrachov theorem (see, e.g. \cite[Thm.~6.2]{Adams}). Combining these facts one can check that the augmented
Lagrangian is indeed weakly sequentially l.s.c. in $(v, \lambda)$. The question of when the augmented Lagrangian is
weakly sequentially l.s.c. in the case of nonholonomic constraints is an interesting open problem.
\end{remark}

Let us also provide sufficient conditions for the complete exactness of the augmented Lagrangian based on
Theorem~\ref{THM:COMPLETE_EXACTNESS}. The use of this theorem allows one to avoid very restrictive and difficult to
verify assumptions on the weak sequential lower semicontinuity of the augmented Lagrangian and the boundedness of the
sublevel set $\Omega_c(\gamma)$ in the space $MW^{n, 2}(E; \mathbb{R}^d)$.

\begin{theorem}
Let the following assumptions be valid:
\begin{enumerate}
\item{for all KKT points $(u_*, \lambda_*)$ of problem \eqref{prob:Nonholonomic} one has 
$u_* \in \overline{u} + MW_0^{n, 2}(E; \mathbb{R}^d)$;}

\item{the function $f$ is twice differentiable in $(u, \xi)$, satisfies the growth conditions
\eqref{eq:MultiDimGrowthCond}, and its second order derivatives are Carath\'{e}odory functions that are essentially
bounded on $\mathbb{R}^d \times \mathbb{R}^{d \times n} \times \mathbb{R}^n$;}

\item{the function $F$ has the form $F(u, \xi, x) = A(x) u + \sum_{i = 1}^n B_i(x) \xi_i + D(x)$ for some essentially
bounded measurable functions $A$, $B_i$ and $D$, and $P_F(\cdot) \ne 0$ (if the functions $B_i$ are continuously
differentiable, it is sufficient to suppose that $A(x) - \sum_{i = 1}^n D_i B_i(x) \ne 0$ on a set of positive
measure);}

\item{for any $\gamma \in \mathbb{R}$ there exists $r = r(\gamma) > 0$ such that the set
\[
  Z_r(\gamma) := \Big\{ u \in W^{1, 2}_0(E; \mathbb{R}^d) \Bigm| 
  \mathcal{I}(\overline{u} + u) + r \| F(u(\cdot), \nabla u(\cdot), \cdot) \|_2^2 \le \gamma \}
\]
is bounded in $W^{1, 2}(E; \mathbb{R}^d)$.}
\end{enumerate}
Then for the augmented Lagrangian $\mathscr{L}(v, \lambda, c)$ to be completely exact on the set $S_c(\gamma)$ for any 
$\gamma \in \mathbb{R}$ it is necessary and sufficient that there exists an unbounded increasing sequence 
$\{ c_k \} \subset (0, + \infty)$ such that for any $k \in \mathbb{N}$ the augmented Lagrangian 
$\mathscr{L}(\cdot, c_k)$ attains a global minimum on the set $X \times L^2(E; \mathbb{R}^{\ell})$.
\end{theorem}

The proof of this theorem largely repeats the proof of Theorem~\ref{thm:Isoperimetric_ExAugmLagr_Stronger} and is based
on verifying the assumptions of Theorem~\ref{THM:COMPLETE_EXACTNESS}. Therefore, we omit its proof for the sake of
shortness.

%

\section{An application to optimal control}
\label{sect:OptimalControl}

Below we use standard definitions and results on control problems for linear evolution equations that can be found in
\cite{TucsnakWeiss}. Let $\mathscr{H}$ and $\mathscr{U}$ be complex Hilbert spaces, $\mathbb{T}$ be a strongly
continuous semigroup on $\mathscr{H}$ with generator $\mathcal{A} \colon \mathcal{D}(A) \to \mathscr{H}$, where
$\mathcal{D}(A)$ is the domain of $A$. Let also $\mathcal{B}$ be an \textit{admissible control operator} for
$\mathbb{T}$ (see \cite[Def.~4.2.1]{TucsnakWeiss}). 

Consider the following fixed-endpoint optimal control problem:
\begin{equation} \label{prob:LinEvolEq}
\begin{split}
  &\min_{(x, u)} \enspace \frac{1}{2} \int_0^T \| u(t) \|_{\mathscr{U}}^2 \, dt
  \\
  &\text{subject to} \enspace \dot{x}(t) = \mathcal{A} x(t) + \mathcal{B} u(t), \quad t \in [0, T], \quad
  x(0) = x_0, \quad x(T) = x_T.
\end{split}
\end{equation}
Here $T > 0$ and $x_0, x_T \in \mathscr{H}$ are fixed, and controls $u(\cdot)$ belong to the space 
$L^2((0, T); \mathscr{U})$.

Our aim is to convert the fixed-endpoint problem \eqref{prob:LinEvolEq} into an equivalent free-endpoint problem for an
exact augmented Lagrangian. To define this augmented Lagrangian, introduce the so-called \textit{input map} 
$F_t u = \int_0^t \mathbb{T}_{t - \tau} \mathcal{B} u(\tau) d \tau$, $t \ge 0$, corresponding to the pair
$(\mathcal{A}, \mathcal{B})$. The assumption that $\mathcal{B}$ is an admissible control operator for $\mathbb{T}$
guarantees that for any $t \in T$ the input map $F_t$ is a bounded linear operator from $L^2((0, T); \mathscr{U})$ to
$\mathscr{H}$ by \cite[Prop.~4.2.2]{TucsnakWeiss}. Moreover, by \cite[Prop.~4.2.5]{TucsnakWeiss} for any 
$x_0 \in \mathscr{H}$ the initial value problem
\begin{equation} \label{eq:LEE_InitValProb}
  \dot{x}(t) = \mathcal{A} x(t) + \mathcal{B} u(t), \quad x(0) = x_0
\end{equation}
has a unique solution $x \in C([0, T]; \mathscr{H})$ given by
\begin{equation} \label{eq:LEE_ExplicitSol}
  x(t) = \mathbb{T}_t x_0 + F_t u \quad \forall t \in [0, T].
\end{equation}
Therefore, problem \eqref{prob:LinEvolEq} can be rewritten as follows:
\begin{equation} \label{prob:LinEvolEq_Equiv}
\begin{split}
  \min_{u \in L^2((0, T); \mathscr{U})} \enspace \frac{1}{2} \int_0^T \| u(t) \|_{\mathscr{U}}^2 \, dt
  \quad \text{subject to} \quad F_T u = \widehat{x}_T,
\end{split}
\end{equation}
where $\widehat{x}_T = x_T - \mathbb{T}_T x_0$. An exact augmented Lagrangian for this problem has the form
\begin{align*}
  \mathscr{L}(u, \lambda, c) = &\frac{1}{2} \int_0^T \| u(t) \|_{\mathscr{U}}^2 \, dt 
  + \langle \lambda, F_T u - \widehat{x}_T \rangle
  \\
  &+ \frac{c}{2} ( 1 + \| \lambda \|_{\mathscr{H}}^2 ) \big\| F_T u - \widehat{x}_T \big\|^2
  + \frac{1}{2} \Big\| F_T \big( u + F_T^* \lambda \big) \Big\|^2,
\end{align*}
where $\lambda \in \mathscr{H}$ is a Lagrange multiplier corresponding to the end-point constraint $x(T) = x_T$ and
$F_T$ is the adjoint operator of $F_T$.

\begin{remark}
Note that to compute $F_T v$ for some $v \in L^2((0, T); \mathscr{U})$ one needs to find a solution of the initial value
problem
\[
  \dot{x}(t) = \mathcal{A} x(t) + \mathcal{B} v(t), \quad x(0) = 0
\]
and set $F_T v = x(T)$. Let us also point out that $(F_T^* \lambda)(t) = \mathcal{B}^* \mathbb{T}_{T - t}^* \lambda$ for
any $t \in [0, T]$ and $\lambda \in \mathscr{H}$. If the operator $\mathcal{A}$ is self-adjoint, then 
$\mathbb{T}_t = \mathbb{T}_t^*$ for all $t \ge 0$ and $\mathbb{T}_{T - t}^* \lambda = y(T - t)$ for any $t \in [0, T]$,
where $y$ is a
solution of the initial value problem $\dot{y}(t) = \mathcal{A} y(t)$, $y(0) = \lambda$ (see
\eqref{eq:LEE_ExplicitSol}).
\end{remark}

Let us provide sufficient conditions for the exactness of the augmented Lagrangian $\mathscr{L}(u, \lambda, c)$, which
ensure that the problem
\begin{equation} \label{prob:OptControl_AugmLagr}
  \min_{(u, \lambda) \in L^2((0, T); \mathscr{U}) \times \mathscr{H}} \mathscr{L}(u, \lambda, c)
\end{equation}
is equivalent to problem \eqref{prob:LinEvolEq} for any $c$ greater than some $c_* > 0$. Note that problem
\eqref{prob:OptControl_AugmLagr} can be rewritten as the following free-endpoint optimal control problem:
\begin{align*}
  &\min_{(x, u, \lambda)} \enspace 
  \begin{aligned}[t]
    &\frac{1}{2} \int_0^T \| u(t) \|_{\mathscr{U}}^2 \, dt + \langle \lambda, x(T) - x_T \rangle
    + \frac{c}{2} (1 + \| \lambda \|_{\mathscr{H}}^2) \| x(T) - x_T \|^2
    \\
    &+ \frac{1}{2} \Big\| F_T \big( u + F_T^* \lambda \big) \Big\|^2
  \end{aligned}
  \\
  &\text{subject to} \enspace \dot{x}(t) = \mathcal{A} x(t) + \mathcal{B} u(t), \quad t \in [0, T], \quad
  x(0) = x_0.
\end{align*}
To conveniently formulate sufficient conditions for the equivalence of this problem and problem \eqref{prob:LinEvolEq},
recall that system \eqref{eq:LEE_InitValProb} is called \textit{exactly controllable} using $L^2$-controls (see, e.g.
\cite{GugatZuazua}), if for any initial state $x_0 \in \mathscr{H}$ and any final state $x_T \in \mathscr{H}$ one can
find a control $u \in L^2((0, T); \mathscr{U})$ such that the corresponding solution $x(\cdot)$ of the initial value
problem \eqref{eq:LEE_InitValProb} satisfies the equality $x(T) = x_T$. Particular examples of exactly controllable
systems can be found in \cite{TucsnakWeiss,BardoLebeauRauch,GugatZuazua}.

\begin{theorem}
Let system \eqref{eq:LEE_InitValProb} be exactly controllable using $L^2$-controls. Then the augmented Lagrangian
$\mathscr{L}(u, \lambda, c)$ is globally exact and for any $\gamma \in \mathbb{R}$ there exists $c(\gamma) > 0$ such
that for all $c \ge c(\gamma)$ the following statements hold true:
\begin{enumerate}
\item{the sublevel set $S_c(\gamma)$ is bounded;}

\item{if a point $(u_*, \lambda_*) \in S_c(\gamma)$ is a stationary point of $\mathscr{L}(u, \lambda, c)$, then $u_*$ is
a globally optimal solution of the problem \eqref{prob:LinEvolEq_Equiv} and $(u_*, \lambda_*)$ is a KKT-point of this
problem.
}
\end{enumerate}
\end{theorem}

\begin{proof}
In the case of problem \eqref{prob:LinEvolEq_Equiv} for any $u \in L^2((0, T); \mathscr{U})$ the function $Q(u)$ has
the form $Q(u)[\lambda] = 0.5 \| F_T(F_T^* \lambda) \|^2$ and does not depend on $u$. Note that from the expression
\eqref{eq:LEE_ExplicitSol} for the solution of the initial value problem \eqref{eq:LEE_InitValProb} it follows that 
the exact controllability of system \eqref{eq:LEE_InitValProb} is equivalent to the surjectivity of the input map $F_T$.
Therefore, by the assumption of the theorem and \cite[Lemma~3.3]{Dolgopolik_ExAugmLagr} the function $Q(\cdot)$ is
uniformly positive definite on the entire space $L^2((0, T); \mathscr{U})$.

Since the objective function of the problem \eqref{prob:LinEvolEq_Equiv} is obviously coercive, the sublevel set
$\Omega_c(\gamma)$ is bounded for any $c \ge 0$ and $\gamma \in \mathbb{R}$. Hence applying
\cite[Cor.~3.6]{Dolgopolik_ExAugmLagr} one obtains that the sublevel set $S_c(\gamma)$ is bounded for any sufficiently
large $c = c(\gamma)$.

Observe that the augmented Lagrangian $\mathscr{L}(u, \lambda, c)$ can be represented as the sum of the convex in 
$(u, \lambda)$ function
\[
  \frac{1}{2} \int_0^T \| u(t) \|_{\mathscr{U}}^2 \, dt 
  + \langle \lambda, F_T u - \widehat{x}_T \rangle 
  + \frac{c}{2} \| F_T u - \widehat{x}_T \|^2
  + \frac{1}{2} \Big\| F_T \big( u + F_T^* \lambda \big) \Big\|^2
\]
and the function $\omega(u, \lambda) = \| \lambda \|_{\mathscr{H}}^2 \| F_T u - \widehat{x}_T \|^2$, which, as one can
readily check, is weakly sequentially l.s.c. as the product of two \textit{nonnegative} convex (and therefore
sequentially l.s.c.) functions. Consequently, for any $c \ge 0$ the augmented Lagrangian $\mathscr{L}(\cdot, c)$ is
weakly sequentially lower semicontinuous. Hence applying \cite[Thm.~5.3]{Dolgopolik_ExAugmLagr} one can conclude that
$\mathscr{L}(u, \lambda, c)$ is globally exact. 

Finally, applying \cite[Thm.~5.7]{Dolgopolik_ExAugmLagr} one obtains that for any sufficiently large $c$ stationary
points $(u_*, \lambda_*) \in S_c(\gamma)$ of $\mathscr{L}(u, \lambda, c)$ are precisely KKT-point of the problem
\eqref{prob:LinEvolEq_Equiv}. The convexity of this problem implies that any such $u_*$ is globally optimal, that is,
the last statement of the theorem holds true.
\end{proof}

\begin{remark}
Let us note that the assumption on the exact controllability of system \eqref{eq:LEE_InitValProb} was used in closely 
related results on the exactness of a nonsmooth penalty function for problem \eqref{prob:LinEvolEq}. Namely, in
\cite{GugatZuazua} it was shown that for any sufficiently large $c \ge 0$ problem \eqref{prob:LinEvolEq} is equivalent
to the free-endpoint penalized problem
\begin{align*}
  &\min_{(x, u)} \enspace \frac{1}{2} \int_0^T \| u(t) \|_{\mathscr{U}}^2 \, dt + c \| x(T) - x_T \|_{\mathscr{H}}
  \\
  &\text{subject to} \enspace \dot{x}(t) = \mathcal{A} x(t) + \mathcal{B} u(t), \quad t \in [0, T], \quad x(0) = x_0
\end{align*}
under the assumption that system \eqref{eq:LEE_InitValProb} is exactly controllable and there exists $C > 0$ such that
the control from the definition of exact controllability satisfies the inequality 
$\| u \|_{L^2((0, T); \mathscr{U})} \le C (\| x_0 \| + \| x_T \|)$. In turn, in \cite{Dolgopolik_ExPenOptControl_II} it
was shown the the assumption on the existence of $C$ can be dropped and, furthermore, the exact controllability
assumption can be replaced with the weaker assumption on the closedness of the image of the input map $F_T$. Whether the
closedness of the image of this map is sufficient for the global exactness of the augmented Lagrangian 
$\mathscr{L}(u, \lambda, c)$ is an open problem.
\end{remark}

%

%
%
%

\bibliographystyle{abbrv}  
\bibliography{ExAugmLagr_Appl_bibl}

\begin{thebibliography}{10}

\bibitem{Adams}
R.~A. Adams.
\newblock {\em Sobolev Spaces}.
\newblock Academic Press, New York, 1975.

\bibitem{AppellZabrejko}
J.~Appell and P.~P. Zabrejko.
\newblock {\em Nonlinear Superposition Operators}.
\newblock Cambrdige University Press, New York, 1990.

\bibitem{BardoLebeauRauch}
C.~Bardos, G.~Lebeau, and J.~Rauch.
\newblock Sharp sufficient conditions for the observation, control, and
  stabilization of waves from the boundary.
\newblock {\em SIAM J. Control Opim.}, 30:1024--1065, 1992.

\bibitem{Dacorogna_WeakCont}
B.~Dacorogna.
\newblock {\em Weak Continuity and Weak Lower Semicontinuity of Non-Linear
  Functionals}.
\newblock Springer-Verlag, Berlin, Heidelberg, 1982.

\bibitem{Dacorogna}
B.~Dacorogna.
\newblock {\em Direct Methods in the Calculus of Variations}.
\newblock Springer, New York, 2007.

\bibitem{Demyanov2003}
V.~F. Demyanov.
\newblock Constrained problems of calculus of variations via penalization
  technique.
\newblock In P.~Daniele, F.~Giannessi, and A.~Maugeri, editors, {\em
  Equilibrium Problems and Variational Models}, pages 79--108. Springer,
  Boston, 2003.

\bibitem{Demyanov2004}
V.~F. Dem'yanov.
\newblock Exact penalty functions and problems of variation calculus.
\newblock {\em Autom. Remote Control}, 65:280--290, 2004.

\bibitem{Demyanov2005}
V.~F. Demyanov.
\newblock An old problem and new tools.
\newblock {\em Optim. Methods Softw.}, 20:53--70, 2005.

\bibitem{DemyanovGiannessi2003}
V.~F. Demyanov and F.~Giannessi.
\newblock Variational problems with constraints involving higher-order
  derivatives.
\newblock In P.~Daniele, F.~Giannessi, and A.~Maugeri, editors, {\em
  Equilibrium Problems and Variational Models}, pages 109--134. Springer,
  Boston, 2003.

\bibitem{DemyanovTamasyan2011}
V.~F. Demyanov and G.~S. Tamasyan.
\newblock Exact penalty functions in isoperimetric problems.
\newblock {\em Optim.}, 60:153--177, 2011.

\bibitem{DemyanovTamasyan2014}
V.~F. Demyanov and G.~S. Tamasyan.
\newblock Direct methods in the parametric moving boundary variational problem.
\newblock {\em Numer. Funct. Anal. Optim.}, 35:934--961, 2014.

\bibitem{Dolgopolik_MultidimCalcVar}
M.~V. Dolgopolik.
\newblock New direct numerical methods for some multidimensional problems of
  the calculus of variations.
\newblock {\em Numer. Funct. Anal. Optim.}, 39:467--490, 2018.

\bibitem{Dolgopolik_ExPenOptControl_II}
M.~V. Dolgopolik.
\newblock Exact penalty functions for optimal control problems {I}{I}: {E}xact
  penalization of terminal and pointwise state constraints.
\newblock {\em Optim. Control Appl. Methods}, 41:898--947, 2020.

\bibitem{Dolgopolik_ExAugmLagr}
M.~V. Dolgopolik.
\newblock Exact augmented {L}agrangians for constrained optimization problems
  in {H}ilbert space {I}: theory.
\newblock {\em Optim.}, 73:1355--1395, 2024.

\bibitem{GiaquintaHildebrandt}
M.~Giaquinta and S.~Hildebrandt.
\newblock {\em Calculus of Variations {I}}.
\newblock Springer-Verlag, Beglin, Heidelberg, 2004.

\bibitem{GugatZuazua}
M.~Gugat and E.~Zuazua.
\newblock Exact penalization of terminal constraints for optimal control
  problems.
\newblock {\em Optim. Control Appl. Methods}, 37:1329--1354, 2016.

\bibitem{Leoni}
G.~Leoni.
\newblock {\em A first course in {S}obolev spaces}.
\newblock American Mathematical Society, Providence, RI, 2009.

\bibitem{TucsnakWeiss}
M.~Tucsnak and G.~Weiss.
\newblock {\em Observation and Control for Operator Semigroups}.
\newblock Birkh\"{a}user, Basel, 2009.

\end{thebibliography}

\section*{Appendix. The proof of Theorem~\ref{THM:COMPLETE_EXACTNESS}}

We divide the proof of Theorem~\ref{THM:COMPLETE_EXACTNESS} into three lemmas.

\begin{lemma} \label{lem:BoundedMultipliers}
Under the assumptions of Theorem~\ref{THM:COMPLETE_EXACTNESS} there exist $c_* > 0$ and a bounded set 
$\Lambda \subset H \times \mathbb{R}^m$ such that for any $c \ge c_*$ one has
$S_c(\gamma) \subset \Omega_c(\gamma + \varepsilon) \times \Lambda$.
\end{lemma}

\begin{proof}
The fact that there exists $\widehat{c} > 0$ such that for any $c \ge \widehat{c}$ and all 
$(x, \lambda, \mu) \in S_c(\gamma)$ one has $x \in \Omega_c(\gamma + \varepsilon)$ follows directly from 
\cite[Lemma~3.2]{Dolgopolik_ExAugmLagr} and is proved in exactly the same way as in the proof of 
\cite[Thm.~3.5]{Dolgopolik_ExAugmLagr}.

Arguing by reductio ad absurdum, suppose that the claim of the lemma is false. Then for any $n \in \mathbb{N}$ there
exists $(x_n, \lambda_n, \mu_n) \in S_n(\gamma)$ such that $\| \lambda_n \| + |\mu_n| \ge n$. As was noted in the proof
of \cite[Thm.~3.5]{Dolgopolik_ExAugmLagr}, one has
\[
  \eta(x, \lambda, \mu) = Q(x)[\lambda, \mu] + \langle Q_{1, \lambda}(x), \lambda \rangle
  + \langle Q_{1, \mu}(x), \mu \rangle + Q_0(x)
\]
for some functions $Q_{1, \lambda}(x)$, $Q_{1, \mu}(x)$, and $Q_0(x)$ defined via the functions $f$, $g_i$, $F$, as
well as their first order derivatives. Therefore, by the second assumption of Theorem~\ref{THM:COMPLETE_EXACTNESS} there
exists $C > 0$ such that for any $n \ge r$ one has
\[
  \| Q_{1, \lambda}(x_n) \| \le C, \quad |Q_{1, \mu}(x)| \le C, \quad |Q_0(x_n)| \le C
\]
since $x_n \in \Omega_n(\gamma + \varepsilon) \subseteq \Omega_r(\gamma + \varepsilon)$. Consequently, by the third
assumption of Theorem~\ref{THM:COMPLETE_EXACTNESS} for any $n \ge r$ one has
\[
  \eta(x_n, \lambda_n, \mu_n) \ge \sigma \big( \| \lambda_n \|^2 + |\mu_n|^2 \big)
  - C \big( \| \lambda_n \| + |\mu_n| \big) - C.
\]
Hence with the use of \cite[Lemma~3.2]{Dolgopolik_ExAugmLagr} one gets that
\begin{align*}
  \mathscr{L}(x_n, \lambda_n, \mu_n, n) &\ge f(x_n) + \frac{n}{2} \phi(\| F(x_n) \|^2)
  \\
  &+ \frac{n}{2 \psi(0)} \big| \max\{ g(x_n), 0 \} \big|^2 - \frac{1}{2n \phi'(0)} 
  - \frac{1}{2}(1 + m) \frac{\psi(0)}{n}
  \\
  &+ \sigma \big( \| \lambda_n \|^2 + |\mu_n|^2 \big) - C \big( \| \lambda_n \| + |\mu_n| \big) - C
\end{align*}
for all $n \ge r$. Hence taking into account the facts that $x_n \in \Omega_n(\gamma + \varepsilon)$, the function $f$
is bounded on $\Omega_n(\gamma + \varepsilon)$ and $\| \lambda_n \| + |\mu_n| \ge n$ one can conclude that 
$\mathscr{L}(x_n, \lambda_n, \mu_n, n) \to + \infty$ as $n \to \infty$, which is impossible, since
$(x_n, \lambda_n, \mu_n) \in S_n(\gamma)$. 
\end{proof}

\begin{lemma} \label{lem:GradientInequal}
Under the assumptions of Theorem~\ref{THM:COMPLETE_EXACTNESS} for any $K > 0$ there exists $c_* > 0$ such that
\begin{equation} \label{eq:GradientInequal}
  \Big\| \nabla \mathscr{L}(x, \lambda, \mu, c) \Big\| \ge 
  K \left( \| F(x) \| + \left| \max\left\{ g(x), - \frac{p(x, \mu)}{c} \mu \right\} \right| \right)
\end{equation}
for all $c \ge c_*$ and $(x, \lambda, \mu) \in S_c(\gamma)$.
\end{lemma}

\begin{proof}
First note that the claim of the lemma coincides with the claim of \cite[Thm.~4.1]{Dolgopolik_ExAugmLagr} with the only
difference being the fact that in this lemma we do note assume that the set $V = \Omega_r(\gamma + \varepsilon)$ is
bounded. Our aim is to show that the assumption on the boundedness of this set is, in fact, redundant.

Indeed, the proof of \cite[Thm.~4.1]{Dolgopolik_ExAugmLagr} consists of three auxiliary lemmas
\cite[Lemmas~4.2--4.4]{Dolgopolik_ExAugmLagr}. The assumptions on the boundedness of the set $V$ is not used in the
proofs of \cite[Lemmas~4.2 and 4.3]{Dolgopolik_ExAugmLagr} (only the boundedness of $f$, $g_i$, and $F$ as well as their
first and second order derivatives on $V$ is used). Therefore, by \cite[Lemma~4.3]{Dolgopolik_ExAugmLagr} for any 
$K > 0$ and any bounded set $\Lambda \subset H \times \mathbb{R}^m$ there exists $c_* \ge r$ such that inequality
\eqref{eq:GradientInequal} holds true for all $c \ge c_*$ and 
$(x, \lambda, \mu) \in \Omega_c(\gamma + \varepsilon) \times \Lambda$. Hence, increasing $c_*$, if necessary, and
applying Lemma~\ref{lem:BoundedMultipliers} one obtains that inequality \eqref{eq:GradientInequal} holds true
for all $c \ge c_*$ and $(x, \lambda, \mu) \in S_c(\gamma)$.
\end{proof}

\begin{lemma}
Theorem~\ref{THM:COMPLETE_EXACTNESS} holds true.
\end{lemma}

\begin{proof}
Let $c_* > 0$ be from Lemma~\ref{lem:GradientInequal} with $K = 1$ and $n \in \mathbb{N}$ be such that $c_n \ge c_*$. 
By our assumption $\mathscr{L}(\cdot, c_n)$ attains a global minimum on $X \times H \times \mathbb{R}^m$ at some point
$(x_n, \lambda_n, \mu_n)$. By \cite[Prop.~3.1]{Dolgopolik_ExAugmLagr} the function $\mathscr{L}(\cdot, c_n)$ is
Fr\'{e}chet differentiable on its effective domain, which implies that 
$\nabla \mathscr{L}(x_n, \lambda_n, \mu_n, c_n) = 0$. Hence applying Lemma~\ref{lem:GradientInequal} one gets that the
point $x_n$ is feasible for the problem $(\mathcal{P})$ and
\begin{align*}
  \min_{(x, \lambda, \mu)} \mathscr{L}(x, \lambda, \mu, c_n) = \mathscr{L}(x_n, \lambda_n, \mu_n, c_n)
  = f(x_n) + \eta(x_n, \lambda_n, \mu_n) \ge f(x_n) \ge f_*,
\end{align*}
where $f_*$ is the optimal value of the problem $(\mathcal{P})$. Consequently, by
\cite[Lemma~5.2]{Dolgopolik_ExAugmLagr} the augmented Lagrangian $\mathscr{L}(x, \lambda, \mu, c)$ is globally exact,
that is, the first and second statements of Theorem~\ref{THM:COMPLETE_EXACTNESS} holds true.

The proof of the third and fourth statements of Theorem~\ref{THM:COMPLETE_EXACTNESS} coincides with the proof of 
\cite[Thm.~5.7]{Dolgopolik_ExAugmLagr} with the only difference being the fact that instead of
\cite[Thm.~4.1]{Dolgopolik_ExAugmLagr} one needs to apply Lemma~\ref{lem:GradientInequal}.
\end{proof}

\end{document}